\documentclass[a4paper]{article}
\usepackage[utf8]{inputenc}

\usepackage{amsmath,amsthm,amssymb}
\usepackage{geometry}
\usepackage{color}
\usepackage{cite}

\usepackage{tikz}
\usepackage{pgfplots}
\usepackage{subdepth}
\usepackage{enumerate}
\usepackage{authblk}
\usepackage{mathdots}
\usepackage{mathtools}
\usepackage{filecontents}

\usepackage{algorithm}
\usepackage{algpseudocode}

\usepackage{verbatim}
\usepackage{hyperref}

\renewcommand{\top}{\mathsf{T}}
\newcommand{\ctop}{\mathsf{H}}

\newtheorem{lemma}{Lemma}
\newtheorem{theorem}[lemma]{Theorem}
\newtheorem{proposition}[lemma]{Proposition}

\theoremstyle{remark}
\newtheorem{remark}[lemma]{Remark}

\theoremstyle{definition}
\newtheorem{definition}[lemma]{Definition}

\DeclareMathOperator{\vecop}{vec}
\DeclareMathOperator{\re}{re}
\DeclareMathOperator{\im}{im}
\newcommand{\C}{\mathbb C}
\newcommand{\R}{\mathbb R}
\newcommand{\N}{\mathbb N}
\newcommand{\cS}{{\cal S}}
\newcommand{\mS}{{\mathbb S}}
\newcommand{\cQ}{{\cal Q}}
\newcommand{\cI}{{\cal I}}
\newcommand{\la}{{\lambda}}

\newcommand{\ci}{{\mathfrak i}}

\newcommand{\wt}{\widetilde}
\newcommand{\wh}{\widehat}

\DeclareMathOperator{\fl}{\mathsf{fl}}

\DeclarePairedDelimiter{\norm}{\lVert}{\rVert}
\renewcommand{\tilde}{\widetilde}

\title{Nonsingular systems of generalized Sylvester equations: an algorithmic approach\thanks{
    This work was partially
supported by the Ministerio de Econom\'{i}a y Competitividad of Spain
through grants MTM2015-68805-REDT, and MTM2015-65798-P  (F. De Ter\'an), by an INdAM/GNCS Research Project 2016 (B.~Iannazzo, F.~Poloni, and L.~Robol), and by the Research project of the Universit\`a di Perugia ``Soluzione numerica di problemi di algebra lineare strutturata'' (B.~Iannazzo). Part of this work was done during a visit of the first author to the Universit\`a di Perugia as a Visiting Researcher.}}

\author[1]{Fernando De Ter\'an}
\author[2]{ Bruno Iannazzo}
\author[3]{Federico Poloni}
\author[4,5]{Leonardo Robol}

\affil[1]{Departamento de Matem\'aticas, Universidad Carlos III de Madrid, Avda. Universidad 30, 28911 Legan\'es,
Spain. {\tt fteran@math.uc3m.es}.}
\affil[2]{Dipartimento di Matematica e Informatica, Universit\`a di Perugia, Via Vanvitelli 1, 06123 Perugia, Italy. {\tt bruno.iannazzo@dmi.unipg.it}.}
\affil[3]{Dipartimento di Informatica, Universit\`a di Pisa, Largo B. Pontecorvo 3, 56127 Pisa, Italy. {\tt federico.poloni@unipi.it}.}
\affil[4]{Dipartimento di Matematica, Universit\`a di Pisa, Largo B. Pontecorvo 5, 56127 Pisa, Italy. {\tt leonardo.robol@unipi.it}.}
\affil[5]{Institute of Information Science and Technologies ``A. Faedo'', ISTI-CNR, Via G. Moruzzi, 1, 56124 Pisa, Italy.
}

\begin{document}
	
	\date{}
	\maketitle  
	
	\begin{abstract}
		We consider the uniqueness of solution (i.e., nonsingularity) of
		systems of $r$ generalized Sylvester and $\star$-Sylvester
		equations with $n\times n$ coefficients. After several reductions, we show that it is sufficient to analyze periodic systems having, at most, one
		generalized $\star$-Sylvester equation. We provide characterizations for the nonsingularity in terms of spectral properties of
		either matrix pencils or formal matrix products, both
		constructed from the coefficients of the system. The proposed
		approach uses the periodic Schur decomposition, and leads to a backward stable 
		$O(n^3r)$ algorithm for computing the (unique) solution. \\[12pt]
	\end{abstract}

	\section{Introduction}\label{sec:intro}

The {\em generalized Sylvester} equation
\begin{equation}\label{eq:gensylv}
	 AXB-CXD=E,
\end{equation}
goes back to, at least, the early 20th century \cite{wedder04}. Here the unknown $X$, the coefficients $A,B,C,D$, and the right-hand side $E$ are complex matrices of appropriate size. This equation has attracted much attention since the 1970s, mainly due to its appearance in applied problems (see, for instance, \cite{mitra,kosir96,kmnt14,chu87,simoncini16}). 

Another related equation, whose interest is growing recently (see, for instance, \cite{ccl12,dd11,ddgmr,brunofernando,rectangular,dgks}), arises when introducing the $\star$ operator in the second appearance of the unknown. This equation is the {\em generalized $\star$-Sylvester equation}
\begin{equation}\label{eq:genstarsylv}
	AXB-CX^{\star}D=E,
\end{equation}
where the unknown $X$, the coefficients $A,B,C,D$, and the right-hand side $E$ are again complex matrices of appropriate size, and $\star$ can be either the transpose ($\top$) or the conjugate transpose ($\ctop$) operator. When $\star=\top$, the equation can be seen as a linear system in the entries of the unknown $X$, while if $\star=\ctop$, the equation is no more linear in the entries of $X$ because of conjugation. Nevertheless, with the usual isomorphism $\C\cong\R^2$, obtained by splitting the real and imaginary parts, it turns out to be a linear system with respect to the real entries of $\re(X)$ and $\im(X)$.

One could argue that, in some sense, solving generalized Sylvester and
$\star$-Sylvester equations is an elementary problem both from the
theoretical and the computational point of view, since they are
equivalent to linear systems. Nevertheless, there has been great
interest in giving conditions on the existence and uniqueness of
solutions based just on properties of certain small-sized matrix
pencils constructed from the coefficients. For instance, when all coefficients are square, it is known
that \eqref{eq:gensylv} has a unique solution if and only if the two
pencils $A-\la C$ and $D-\la B$ have disjoint spectra
\cite[Th. 1]{chu87}, whereas the uniqueness of solutions of
\eqref{eq:genstarsylv} depends on spectral properties of the matrix
pencil
$\left[\begin{smallmatrix}\la D^\star&B^\star\\A&-\la
    C\end{smallmatrix}\right]$ (see \cite[Th. 15]{brunofernando}).

On the other hand, if all coefficients are square and of size $n$, then the resulting linear system has size $n^2$ or $2n^2$. {From} the computational point of view, solving a linear system of size $n^2$ with standard (non-structured) algorithms may be prohibitive, since they result in a method which approximates the solution in $O(n^6)$ (floating point) arithmetic operations (flops). However, dealing with the coefficients it is possible to get algorithms requiring only $O(n^3)$ flops, such as the one given in \cite{chu87}.

Recently, systems of coupled generalized Sylvester and $\star$-Sylvester equations have been considered, and useful conditions on the existence of solutions have been derived in \cite{dk16}. Here, we consider the same kind of systems and provide further characterizations for the uniqueness of their solution, for any right-hand side, based on certain spectral conditions on their coefficients. It is worth to emphasize that, while in \cite{dk16} non-square coefficients are allowed, as long as the matrix products are well-defined, here we assume that all coefficients, as well as the unknowns, are square of size $n\times n$. This choice has been made because the problem of nonsingularity, even for just one equation, presents certain additional subtleties when the coefficients are not square or they are square with different sizes (see~\cite{rectangular}). In the assumption that all coefficients are square and of size $n\times n$, such a system of matrix equations is equivalent to a square linear system, which has a unique solution, for any right-hand side, if and only if the coefficient matrix is nonsingular. For this reason, we will use the term \emph{nonsingular system} as a synonym of a system having a unique solution (for any right-hand side).

The {\em systems of generalized Sylvester and $\star$-Sylvester equations} that we consider are of the form
\begin{equation}\label{eq:gensystem}
\begin{array}{l}
A_kX_{\alpha_k}^{s_k}B_k-C_kX_{\beta_k}^{t_k}D_k=E_k,\qquad k=1,\ldots,r,
\end{array}
\end{equation}
where all matrices involved are complex and of size $n\times n$, the indices $\alpha_i,\beta_i$ of the unknowns are positive integers and can be equal or different to each other, and $s_i,t_i\in\{1,\star\}$.

Our approach starts by reducing the problem on the nonsingularity of \eqref{eq:gensystem} to the special case of {\em periodic} systems of the form
\begin{equation}\label{eq:periodic}
\left\{
\begin{array}{cccc}
A_kX_kB_k-C_kX_{k+1}D_k&=&E_k,& k=1,\ldots,r-1,\\
A_rX_rB_r-C_rX_1^sD_r&=&E_r,
\end{array}
\right.
\end{equation}
where $s\in\{1,\star\}$. We  provide an explicit characterization of nonsingularity only for periodic systems like \eqref{eq:periodic}. However, our reduction allows one to get a characterization for any system like \eqref{eq:gensystem} after undoing all changes that take the system \eqref{eq:gensystem} into \eqref{eq:periodic}. Since these systems can be seen as linear systems with a square matrix coefficient, the criteria for nonsingularity do not depend on the right-hand sides $E_k$, but only on the coefficients $A_k,B_k,C_k,D_k$, for $k=1,\dots,r$.

Periodic systems of Sylvester equations naturally arise in the context of discrete-time periodic systems, and they have been analyzed by several authors (see, for instance, \cite{agjk08,gk06,gkg07,vvd01}). Prior to our work, Byers and Rhee provided in the unpublished work \cite{byers-rhee} a characterization for the nonsingularity of \eqref{eq:periodic} with $s=1$, together with an $O(n^3r)$ algorithm to compute the solution. 

The first contribution of the present work is the reduction of a nonsingular system of Sylvester and $\star$-Sylvester equations \eqref{eq:gensystem} to several disjoint systems of periodic type \eqref{eq:periodic}, where all equations are generalized Sylvester, with the exception of the last one that may be either a generalized Sylvester or a generalized $\star$-Sylvester equation. We note that neither the coefficients, nor the number of equations in the original and the reduced system necessarily coincide.

As a second contribution, we provide a characterization for the nonsingularity of \eqref{eq:periodic} for $s=\ctop,\top$ (i.\@ e., $s=\star$, according to our notation). This characterization appears in two different formulations. The first one is given in terms of the spectrum of {\em formal products} constructed from the coefficients of the system (we include the case $s=1$, treated in Theorem \ref{thm:1eigentheorem}, and the case $s=\star$, treated in Theorem \ref{thm:stareigentheorem}). The second formulation, valid for $s=\star$, is given in terms of spectral properties of a block-partitioned  $(2rn)\times(2rn)$ matrix pencil constructed in an elementary way from the coefficients (Theorem \ref{thm:pencilcond}). This characterization extends the one in \cite{brunofernando} for the single equation \eqref{eq:genstarsylv}, and it is in the same spirit as the one in \cite{byers-rhee} for periodic systems with $s=1$.

The third contribution of the paper is to provide an $O(n^3r)$ algorithm to compute the unique solution of a nonsingular system. Our algorithm is a Bartels-Stewart like algorithm, based on the periodic Schur form~\cite{bojanczyk1992periodic}. It extends the one in \cite{byers-rhee} for systems of Sylvester equations only, the one in \cite{dd11} for the $\star$-Sylvester equation $AX+X^\star D=E$, and the one outlined in \cite[\S 4.2]{ccl12} for~\eqref{eq:genstarsylv}.

We note that extending the results of~\cite{byers-rhee} to include $\star$-Sylvester equations is not a trivial endeavour: the presence of transpositions creates additional dependencies between the data, hence we need a different strategy to reduce the coefficients to a triangular form, and the resulting criteria have a significantly different form.

Throughout the manuscript, $\ci$ denotes the imaginary unit, that is, $\ci^2=-1$. By $M^{-\star}$ we denote the inverse of the invertible matrix $M^\star$, with $\star=\ctop,\top$. A pencil $\cQ(\lambda)$ is {\em regular}
if it is square and $\det\cQ(\lambda)$ is not identically zero. We use the symbol $\Lambda({\cal Q})$ to denote the \emph{spectrum} of a regular matrix pencil $\cQ(\lambda)$, that is the set of values $\lambda$ such that $\cQ(\lambda)$ is singular (including $\infty$ if the degree of $\det\cQ(\lambda)$ is smaller than the size of the pencil). For simplicity, we use the term {\em system of  Sylvester-like equations} for a system of generalized Sylvester and $\star$-Sylvester equations.

The paper is organized as follows. In Section~\ref{sec:applications} we 
present some applications of systems of Sylvester and $\star$-Sylvester
equations; in Section~\ref{sec:prel} the periodic Schur decomposition and the concept of formal matrix product are recalled. Section~\ref{sec:main} hosts the main theoretical results of the paper, whose proofs are deferred to Section~\ref{sec:existenceuniqueness}, after Sections~\ref{sec:red1} and~\ref{sec:red2}, that are devoted to some successive simplifications of the problem which are useful for the proofs. Section \ref{sec:algorithm} is devoted to describe and analyze an efficient algorithm for the solution of systems of Sylvester-like equations. Finally, in Section~\ref{sec:conc} we draw some conclusions.

%%%%%%%%%%%%%%%%%%%%%%%%%%%%%%%%
%%%%%%%%%%%%%%%%%%%%%%%%%%%%%%%%
%%%%%%%%%%%%%%%%%%%%%%%%%%%%%%%%
%%%%%%%%%%%%%%%%%%%%%%%%%%%%%%%%
%%%%%%%%%%%%%%%%%%%%%%%%%%%%%%%%
%%%%%%%%%%%%%%%%%%%%%%%%%%%%%%%%
%%%%%%%%%%%%%%%%%%%%%%%%%%%%%%%%
%%%%%%%%%%%%%%%%%%%%%%%%%%%%%%%%
%%%%%%%%%%%%%%%%%%%%%%%%%%%%%%%%
%%%%%%%%%%%%%%%%%%%%%%%%%%%%%%%%

\section{Applications} \label{sec:applications}

Sylvester-like equations appear in various fields
of applied mathematics. In some cases, the applications have natural
``periodic extensions'', where systems of these equations come into
play. 

As an example, consider a $2 \times 2$ 
block upper triangular matrix $M=\left[\begin{smallmatrix}
    A & C \\
     0 & B \\
  \end{smallmatrix}\right]$, 
and assume that we want to block diagonalize it, setting $C$ to zero
with a similarity transformation. This problem arises, for instance, when $M$ is the block Schur form of a given matrix and we want to decouple the action of the parts 
of the spectrum contained 
 in $A$ and $B$. Then, we can look for a
matrix $V$ such that 
\begin{equation}\label{vmatrix}
  V^{-1} M V = \begin{bmatrix}
   A &0\\0 & B
  \end{bmatrix}, \qquad 
  V = \begin{bmatrix}
   I & X \\
    0 & I \\
  \end{bmatrix}.
\end{equation}
This problem can be solved by finding a solution to the Sylvester equation $AX - XB + C = 0$, and admits a natural extension in periodic form, when we want to block 
diagonalize the product of $2\times 2$ block upper triangular matrices, as the one arising in a periodic Schur form. We start from
\begin{equation}\label{eq:rev1}
  M = M_1 \cdots M_r, \qquad 
  M_i := \begin{bmatrix}
    A_i & C_i \\ 
    0& B_i \\
  \end{bmatrix}, 
\end{equation}
where the blocks have the same size for each $i$, and we want to block diagonalize $M$. For stability reasons, rather than working directly on the product $M$, 
it is often preferable to look for matrices $V_i$ such that $V_i^{-1} M_i V_{i+1}$
are all block diagonal, with $V_{r+1} = V_1$ (see, e.g. \cite{watkins2005product}). 
If we impose on $V_i=\left[\begin{smallmatrix}I&X_i\\0 &I\end{smallmatrix}\right]$ the same block upper triangular structure
we had for $V$ in \eqref{vmatrix}, then we obtain the periodic system
of Sylvester equations $A_i X_{i+1} - X_i B_i + C_i = 0$, for $i=1,\ldots,r$, with $X_{r+1} = X_1$. 

Similarly, decoupling saddle-point matrices (as quadratic forms) given in product form
\[
  N = N_1 N_2 N_3 = \begin{bmatrix}
    A_1 & 0   \\
    C_1 & B_1 \\
  \end{bmatrix} \begin{bmatrix}
    0 & A_2 \\
    B_2 & C_2 \\
  \end{bmatrix} \begin{bmatrix}
    B_3 & C_3 \\
    0   & A_3 \\
  \end{bmatrix} = \begin{bmatrix}
    0 & A_1 A_2 A_3 \\
    B_1 B_2 B_3 & C_1 A_2 A_3 + B_1 C_2 A_3 + B_1 B_2 C_3 \\
  \end{bmatrix}
\]
(see e.g.~\cite{ReeS18} for similar factorizations)
naturally leads to systems of $\star$-Sylvester equations:
 one can choose the following change of bases to eliminate the blocks $C_i$
\[
U_1^\star NU_1 = 
U_1^\star
\begin{bmatrix}
    A_1 & 0\\
    C_1 & B_1
\end{bmatrix}
V_2^{-1}
V_2
\begin{bmatrix}
    0 & A_2\\
    B_2 & C_2
\end{bmatrix}
V_3 V_3^{-1}
\begin{bmatrix}
    B_3 & C_3\\
    0 & A_3
\end{bmatrix}
U_1,
\]
\[
U_1 = \begin{bmatrix}
    I & X_1\\
    0 & I
\end{bmatrix}, \quad 
V_2 = \begin{bmatrix}
    I   & 0\\
    X_2 & I
\end{bmatrix}, \quad
V_3 = \begin{bmatrix}
    I & X_3\\
    0 & I
\end{bmatrix};
\]
then the factors become
\begin{align*}
U_1^\star
\begin{bmatrix}
    A_1 & 0\\
    C_1 & B_1
\end{bmatrix}
V_2^{-1} &=
\begin{bmatrix}
    A_1 & 0\\
    X_1^\star A_1 - B_1X_2+ C_1 & B_1
\end{bmatrix},\\
V_2
\begin{bmatrix}
    0 & A_2\\
    B_2 & C_2
\end{bmatrix}
V_3 &=
\begin{bmatrix}
    0 & A_2\\
    B_2 & X_2 A_2 + B_2 X_3 + C_2
\end{bmatrix},\\
V_3^{-1}
\begin{bmatrix}
    B_3 & C_3\\
    0 & A_3
\end{bmatrix}
U_1 &=
\begin{bmatrix}
    B_3 & -X_3A_3+ B_3X_1+ C_3\\
    0 & A_3
\end{bmatrix}.
\end{align*}
Hence the blocks in the position of the $C_i$ vanish if the $X_i$ solve the periodic system of $\star$-Sylvester equations
\[
   \left\{
   \begin{array}{l}
   X_1^\star A_1 - B_1X_2+ C_1 = 0,\\
   X_2A_2 + B_2 X_3 + C_2 = 0,\\
   -X_3A_3+ B_3X_1+ C_3 = 0.
   \end{array}
    \right.
\]

Another relevant application is the reordering of periodic Schur forms.
In order to swap the diagonal blocks of $M$ in \eqref{eq:rev1} it may be convenient to swap the blocks of the factors $M_i$, for $i=1,\ldots,r$. While the problem of swapping the blocks of $M$ can be reduced to a Sylvester equation \cite{baidem}, the problem of swapping the blocks of the factors can be reduced to a periodic system of Sylvester equations.
Indeed, swapping diagonal entries of matrices given in products form, without forming the product, is an essential step in the eigenvector recovery procedures
of some fast methods for matrix polynomial eigenvalue problems (see \cite{aurentz2016fast,aurentz2018corebook}).

	\section{Periodic Schur decomposition of formal matrix products}
	\label{sec:prel}

In order to state and prove the nonsingularity results for a system of Sylvester-like equations and to design an efficient algorithm to compute the solution, we need to introduce several results and definitions that extend the ideas of matrix pencils and generalized eigenvalues to products of matrices of an arbitrary number of factors. These are standard tools in the literature (see, for instance, \cite{gk06,gkg07}). 

	\begin{theorem}[Periodic Schur decomposition
		\cite{bojanczyk1992periodic}]\label{thm:periodicschur}
		Let $M_k, N_k$, for $k = 1, \ldots, r$, be two sequences of
		$n\times n$ complex matrices. Then there exist unitary matrices $Q_k, Z_k$,
		for $k = 1, \ldots, r$, such that
		\begin{equation} \label{eq:periodicschur}
			Q_{k}^\ctop M_{k} Z_{k} = T_k, \qquad 
			Q_{k}^\ctop N_{k} Z_{k+1} = R_k, \qquad 
			k = 1, \ldots, r    	
		\end{equation}
		where $T_k, R_k$ are upper triangular and 
		$Z_{r+1}=Z_1$. 
	\end{theorem}
	If the matrices $N_k$ are invertible, Theorem~\ref{thm:periodicschur} means that we can apply suitable unitary changes of bases to the product
	\begin{equation} \label{eq:formalproduct}
		\Pi = N_r^{-1}M_r N_{r-1}^{-1}M_{r-1}\dotsm N_1^{-1}M_1
	\end{equation}
	to make all its factors upper triangular simultaneously. More precisely,
	\[
Z_{1}^{-1}\Pi Z_1 = R_r^{-1}T_r R_{r-1}^{-1}T_{r-1}\dotsm R_1^{-1}T_1.
	\]
	In this case, the eigenvalues of $\Pi$ are
	\begin{equation} \label{eq:periodiceigenvalues}
		\lambda_i = \frac{ (T_1)_{ii}(T_2)_{ii}\cdots(T_r)_{ii}}{(R_1)_{ii}(R_2)_{ii}\cdots(R_r)_{ii}}, \quad i = 1,2,\dots,n.
	\end{equation}
	Even when some of the $N_k$ matrices are not invertible, we call the expression~\eqref{eq:formalproduct} a \emph{formal matrix product}, and~\eqref{eq:periodicschur} a {\em formal periodic Schur form} of the product. If $(T_1)_{ii}(T_{2})_{ii}\cdots(T_r)_{ii} = (R_1)_{ii}(R_2)_{ii}\cdots(R_r)_{ii} = 0$, for some $i\in\{1,2,\dots,n\}$, we call the formal product \emph{singular}; otherwise, we call it \emph{regular}. If $\Pi$ is regular, it makes sense to consider the ratios $\lambda_i$ defined in~\eqref{eq:periodiceigenvalues}, with the convention that $\frac{a}{0} = \infty$ for $a\neq 0$. We call these ratios the \emph{eigenvalues} of the regular formal matrix product $\Pi$. The set of eigenvalues of $\Pi$ is called, as usual, the {\em spectrum of $\Pi$}, and we denote it by $\Lambda(\Pi)$.

	We also define the eigenvalues of a formal matrix product of the form
	\[
		\widetilde \Pi = M_r N_{r-1}^{-1}M_{r-1}\dotsm N_1^{-1}M_1N_r^{-1}
	\]
	(i.\@ e., one in which the exponent $-1$ appears in the factors
        in \emph{even} positions) by the same
        formula~\eqref{eq:periodiceigenvalues}.

	\begin{remark}\label{continuity-rem}
	For the notion of eigenvalues of formal products to be well defined, one should prove that it does not depend on the choice of the (non-unique) decomposition~\eqref{eq:periodicschur}. If all $N_i$ matrices are nonsingular, then this is evident because they coincide with the eigenvalues obtained by performing the inversions and computing the actual product $\Pi$. If some of the $N_i$ are singular, then we can use a continuity argument to show that the $\lambda_i$ are the limits, as $\varepsilon\to 0$, of the eigenvalues of 
	\[
		(N_r+\varepsilon P_r)^{-1}M_r (N_{r-1}+\varepsilon P_{r-1})^{-1}M_{r-1}\dotsm (N_1+\varepsilon P_1)^{-1}M_1
	\]
	for each choice of the nonsingular matrices $P_1,P_2,\dots,P_r$ that make the factors $N_k+\varepsilon P_k$ invertible, for all $k=1,\hdots,r$ and sufficiently small $\varepsilon>0$.
	\end{remark}

	\begin{lemma}\label{lem:eigs}
		Let $\Pi = M_{1}^{-1} N_1  \cdots M_r^{-1} N_{r}$ 
		be a formal matrix product. Then, the matrix pencil
		\[
		  \mathcal Q(\lambda) := \begin{bmatrix}
		  \lambda M_1 & -N_1\\
		  & \lambda M_2 & \ddots\\
		  & & \ddots & -N_{r-1}\\
		  -N_r & & & \lambda M_r
		  \end{bmatrix}
		\]
		is regular if and only if $\Pi$ is regular. In this case,
		the eigenvalues of $\mathcal Q(\lambda)$
		are the $r$-th roots of the eigenvalues of $\Pi$, with
		the convention that $\sqrt[r]{\infty} = \infty$. 
	\end{lemma}

\begin{proof}
	Let us start by considering the case when $\Pi$ is regular with distinct (simple) eigenvalues, 
	and all matrices $M_i, N_i$ are invertible. Let 
	$\mu\in\C$ be an eigenvalue of $\Pi$, with $v$ a corresponding 
	right eigenvector, and let $\lambda\in\C$ be such that $\lambda^r=\mu$.
	We set $v_1 := v$, and define
	\[
	  v_j := \lambda^{} N_{j-1}^{-1} M_{j-1} v_{j-1}, 
	    \qquad j = 2, \ldots, r. 
	\]
	Then, the relation $\Pi v = \lambda^r v$ implies 
	$\mathcal Q(\lambda) \widehat{v} = 0$, where
	$\widehat{v} := [ v_1^\top \ v_2^\top \ \hdots \ v_r^\top ]^\top$, which
	can be verified by a direct computation. In particular, 
	all the $r$-th distinct roots of $\mu$ are eigenvalues of $\mathcal Q(\lambda)$.
	
	This implies that $q(\lambda) := \det(\mathcal Q(\lambda)) = 
	\det(\Pi - \lambda^r I) \cdot \det(M_1 \cdots  M_r)$, since 
	$\det (M_1 \cdots M_r)$ is the leading coefficient of the degree $nr$ 
	polynomial $\det\mathcal Q(\lambda)$. 
	Let $Q_k^\ctop M_k Z_k = T_k$ and $Q_k^\ctop N_k Z_{k+1} = R_k$ be a 
	periodic Schur decomposition of $\Pi$. Then, we may write
	\[
	  q(\lambda) = \det(T_1^{-1} R_1 \cdots T_r^{-1} R_r  - \lambda^r I) \cdot \det (T_1 \cdots T_r)= 
	    \det(R_1 \cdots R_r - \lambda^r T_1 \cdots T_r), 
	\]
	where we have swapped the factors inside the determinant using the fact
	that all the matrices are upper triangular. Using a continuity argument like the one in Remark \ref{continuity-rem}, we see that the identity $\det(\mathcal Q(\lambda))=\det(R_1 \cdots R_r - \lambda^r T_1 \cdots T_r)=:p(\la)$ also holds
	when some of the $T_i, R_i$ are singular, and even when $\Pi$ has multiple eigenvalues. This proves the second claim in the statement. In addition, 
	$p(\lambda) \equiv 0$ if and only if $T_i, R_j$ have a common diagonal
	zero for some $i,j$. Since $\mathcal Q(\lambda)$ is singular if and only if 
	$p(\lambda) \equiv 0$, this concludes the proof. 
\end{proof}

\section{Main results}\label{sec:main}

 Here we state the characterizations for the nonsingularity of a periodic system ot type~\eqref{eq:periodic} for each of the three possible cases $s\in\{1,\top,\ctop\}$ (the proofs will be given in Section~\ref{sec:existenceuniqueness}).
Later, in Section~\ref{sec:red1}, we will show that these
characterizations are enough to get a characterization of nonsingularity of the general system \eqref{eq:gensystem}.

We recall the following definition.
	
	 \begin{definition}\label{def:reciprocal} {\rm(Reciprocal free and $\ctop$-reciprocal free set \cite{bk06,ksw09})}. Let $\cS$ be a subset of $\C\cup\{\infty\}$. We say that $\cS$ is 
\begin{itemize} 
\item[{\rm(a)}] {\rm reciprocal free} if $\la\neq \mu^{-1},$ for all $\la,\mu\in \cS$; 
\item[{\rm(b)}]  {\rm $\ctop$-reciprocal free} if $\la\neq (\overline \mu)^{-1},$ for all $\la,\mu\in \cS$.
\end{itemize}
This definition includes the values $\la=0,\infty$, with the customary assumption $\la^{-1}=(\overline\la)^{-1}=\infty,0$, respectively.
\end{definition}

For brevity, we will refer to a {\em$\star$-reciprocal free set} to mean either a reciprocal free or a $\ctop$-reciprocal free set.

The characterization comes in two different forms. The first one uses eigenvalues of formal matrix products. More precisely, we have the following results. 
	
	\begin{theorem} \label{thm:1eigentheorem} Let $A_k,B_k,C_k,D_k,E_k\in\C^{n\times n}$, for $k=1,\hdots,r$. The system
	\[
	\left\{
\begin{array}{cccc}
A_kX_kB_k-C_kX_{k+1}D_k&=&E_k,& k=1,\ldots,r-1,\\
A_rX_rB_r-C_rX_1 D_r&=&E_r,
\end{array}
\right.
\]
%The periodic system~\eqref{eq:periodic}, with $s=1$, 
is nonsingular if and only if the two formal matrix products
			\begin{equation}\label{eq:1formalproduct}
				C_r^{-1}A_r C_{r-1}^{-1}A_{r-1} \dotsm C_1^{-1}A_1\quad \text{ and }\quad D_r B_r^{-1} D_{r-1} B_{r-1}^{-1} \dotsm D_1 B_1^{-1}
			\end{equation}
		are regular and they have disjoint spectra. 
	\end{theorem}

	\begin{theorem} \label{thm:stareigentheorem} 
Let $A_k,B_k,C_k,D_k,E_k\in\C^{n\times n}$, for $k=1,\hdots,r$. The system
\[
	\left\{
\begin{array}{cccc}
A_kX_kB_k-C_kX_{k+1}D_k&=&E_k,&k=1,\ldots,r-1,\\
A_rX_rB_r-C_rX_1^\star D_r&=&E_r,
\end{array}
\right.
\]
is nonsingular if and only if the formal matrix product
	\begin{equation} \label{eq:magicproduct}
		\Pi=D_r^{-\star}B_r^\star D_{r-1}^{-\star}B_{r-1}^\star\dotsm D_1^{-\star}B_1^\star C_r^{-1}A_rC_{r-1}^{-1}A_{r-1}\dotsm C_1^{-1}A_1
	\end{equation}
	is regular and
	\begin{itemize}
\item  if $\star=\ctop$, then $\Lambda(\Pi)$ is an $\ctop$-reciprocal-free set,
\item if $\star=\top$, then $\Lambda(\Pi)\setminus\{-1\}$ is a reciprocal-free set, and the multiplicity of $\la=-1$ as an eigenvalue of $\Pi$ is at most $1$.
\end{itemize}
	\end{theorem}

The second characterization involves eigenvalues of matrix pencils. In what follows, the notation ${\mathfrak R}_p$ stands for the set of $p$th roots of unity, namely,
\begin{equation} \label{rootsofunity}
	{\mathfrak R}_p:=\{ e^{2\pi\ci j/p}, \ j = 0, 1,\ldots, p-1 \}.
\end{equation}

The following results are obtained directly from Theorems~\ref{thm:1eigentheorem}
and \ref{thm:stareigentheorem} by means of Lemma~\ref{lem:eigs}.

	 \begin{theorem} \label{thm:pencilcond} Let $A_k,B_k,C_k,D_k,E_k\in\C^{n\times n}$, for $k=1,\hdots,r$. 
	 The system~\eqref{eq:periodic}, with $s=\star$, 	 
is nonsingular if and only if the matrix pencil 
		\begin{equation}\label{eq:pencil}
			\cQ(\lambda) := \begin{bmatrix}
				\lambda A_1 & C_1\\
				& \ddots & \ddots \\
				& & \lambda A_r & C_r \\
				& & & \lambda B_1^\star & D_1^\star \\
				& & & & \ddots & \ddots \\
				& & & & & \ddots & D_{r-1}^\star\\
			 -D_r^\star & & & & & & \lambda B_r^\star \\ 
			\end{bmatrix}
		\end{equation}
		is regular and 
		\begin{itemize}
		\item if $\star=\ctop$, then $\Lambda(\cQ)$ is $\ctop$-reciprocal-free, and
		\item if $\star=\top$, then $\Lambda(\cQ) \setminus {\mathfrak R}_{2r}$ is reciprocal free and the multiplicity of $\xi$, for any $\xi\in{\mathfrak R}_{2r}$, is at most $1$. 
	\end{itemize}
	\end{theorem}

The proof of Theorem~\ref{thm:pencilcond} can be readily obtained by means
of the following result combined with Theorem~\ref{thm:stareigentheorem}.

\begin{lemma} \label{lem:reciprocalfreedom}
	Let $\mathcal S$ be a subset of $\C\cup\{\infty\}$, let $p\in\N$, and define the sets:
	\[
	-\mathcal S:=\{-z\,| \ z\in\mathcal S\},\quad
	\mathcal S^{-1}:=\{z^{-1}\, | \ z\in\mathcal S\},\quad
	\sqrt[p]{\mathcal S} := \{ 
	z \in \C\cup\{\infty\} \ | \ z^p \in \mathcal S
	\}
	\]
	 (we set $\infty^p=\infty,-\infty=\infty,$ and $\infty^{-1}=0,0^{-1}=\infty$). Then the following statements are equivalent:
	 \begin{itemize}
	 \item[(a)] $\mathcal S$ is $\star$-reciprocal free.
	 \item[(b)] $-\mathcal S$ is $\star$-reciprocal free.
	 \item[(c)] $\mathcal S^{-1}$ is $\star$-reciprocal free.
	 \item[(d)] $\sqrt[p]{\mathcal S}$ is $\star$-reciprocal free.
	 \end{itemize}
\end{lemma}

 The equivalence between claims (a) and (d) in Lemma \ref{lem:reciprocalfreedom} can be found in \cite[Lemma 3]{brunofernando} for $p=2$. The extension to arbitrary $p$, as well as the other equivalences, are straightforward.

\noindent {\em Proof of Theorem {\rm\ref{thm:pencilcond}}}.
By Lemma~\ref{lem:eigs}, $\Lambda (\mathcal Q)=\sqrt[2r]{-\Lambda(\Pi^{-1})}=\sqrt[2r]{-\Lambda(\Pi)^{-1}}$, with $\Pi$ as in Theorem~\ref{thm:stareigentheorem} (the second identity is immediate). From this, we also get $\sqrt[2r]{-(\Lambda(\Pi)\setminus\{-1\})^{-1}}=
\sqrt[2r]{-\Lambda(\Pi)^{-1}\setminus\{1\}}=\Lambda (\mathcal Q)\setminus{\mathfrak R}_{2r}$.

Now, Theorem \ref{thm:pencilcond} is an immediate consequence of Theorem \ref{thm:stareigentheorem} and Lemma \ref{lem:reciprocalfreedom}. 
\hfill$\square$

Theorem \ref{thm:pencilcond} is an extension of 
\cite[Th. 15]{brunofernando}, where the case of a single generalized $\star$-Sylvester equation is treated. It also resembles the characterization obtained in \cite[Th. 3]{byers-rhee} for systems of generalized Sylvester equations (i.e., without $\star$). We reproduce this last result here, for completeness.

\begin{theorem}[Byers and Rhee, \cite{byers-rhee}]\label{thm:byersrhee} 
The system~\eqref{eq:periodic}, with $s=1$, is nonsingular if and only if the matrix pencils
\[
\left[\begin{array}{cccc}
\la A_{1}&C_1&&\\
&\la A_{2}&\ddots&\\
&&\ddots&C_{r-1}\\
C_r&&&\la A_r
\end{array}\right]\quad\mbox{and}\quad
\left[\begin{array}{cccc}
\la D_{1}&B_1&&\\
&\la D_{2}&\ddots&\\
&&\ddots&B_{r-1}\\
B_r&&&\la D_r
\end{array}\right]
\]
are regular and have disjoint spectra. 
\end{theorem}

Our strategy to prove Theorems \ref{thm:1eigentheorem} and \ref{thm:stareigentheorem} for periodic systems \eqref{eq:periodic} relies on several steps. First, we use the fact that the system is equivalent to a system with triangular coefficients, as shown in Section~\ref{sec:triangular}. Second, in Section~\ref{sec:backsubstitution}, when $s=1$ or $s=\top$, we transform the system of matrix equations with triangular coefficients to an equivalent linear system that is block upper triangular in a suitable basis (given by an appropriate order of the unknowns). The remaining case $s=\ctop$ is reduced to the case $s=1$ in Section~\ref{sec:linearizing}. Third, we prove in Section~\ref{sec:existenceuniqueness} that the diagonal blocks of the matrix coefficient of the resulting block triangular system are invertible if and only if the conditions in the statement of Theorems \ref{thm:1eigentheorem} and \ref{thm:stareigentheorem} hold. 

%%%%%%%%%%%%%%%%%%%%%%%%%%%%%%%%
%%%%%%%%%%%%%%%%%%%%%%%%%%%%%%%%
%%%%%%%%%%%%%%%%%%%%%%%%%%%%%%%%
%%%%%%%%%%%%%%%%%%%%%%%%%%%%%%%%
%%%%%%%%%%%%%%%%%%%%%%%%%%%%%%%%
%%%%%%%%%%%%%%%%%%%%%%%%%%%%%%%%
%%%%%%%%%%%%%%%%%%%%%%%%%%%%%%%%
%%%%%%%%%%%%%%%%%%%%%%%%%%%%%%%%
%%%%%%%%%%%%%%%%%%%%%%%%%%%%%%%%
%%%%%%%%%%%%%%%%%%%%%%%%%%%%%%%%

% \section{Reducing the problem}
\section{Reducing the problem to periodic systems}\label{sec:red1}

In this section, we are going to show how to reduce the problem of nonsingularity of a general system \eqref{eq:gensystem} to the question on nonsingularity of periodic systems \eqref{eq:periodic} with at most one $\star$ in the last equation.

%%%%%%%%%%%%%%%%%%%%%%%%%%%%%%%%
%%%%%%%%%%%%%%%%%%%%%%%%%%%%%%%%
%%%%%%%%%%%%%%%%%%%%%%%%%%%%%%%%
%%%%%%%%%%%%%%%%%%%%%%%%%%%%%%%%
%%%%%%%%%%%%%%%%%%%%%%%%%%%%%%%%

	\subsection{Reduction to an irreducible system}\label{sec:red-irred}

	We say that the system \eqref{eq:gensystem} of $r$ equations in $s$ unknowns is {\em reducible} if there are $0<k<s$ unknowns appearing only in $0<h<r$ equations and the remaining $s-k$ unknowns appear only in the remaining $r-h$ equations. In other words, a reducible system can be partitioned into two systems with no common unknowns. A system is said to be {\em irreducible} if it is not reducible.  

Let $\mS$ be a system of $r$ ordered equations like \eqref{eq:gensystem}. Let $\{1,\hdots,r\}=\cI_1\cup\cdots\cup \cI_\ell$ be a partition of the set of indices. Then we denote by $\mS(\cI_j)$, for $j=1,\hdots,\ell$, the system of equations comprising the equations with indices in $\cI_j$.

	\begin{proposition}\label{prop:partition}
Let $\mS$ be a system \eqref{eq:gensystem} with $r$ equations. There exists a partition $\mathcal I_1\cup\cdots\cup\mathcal I_\ell$ of $\{1,\ldots,r\}$ such that, for each $j=1,\ldots,\ell$, the system $\mS(\cI_j)$ is irreducible.
	\end{proposition}

	\begin{proof}
	We proceed by strong induction on $r$. If $r=1$ the system has only one equation and thus it is irreducible. Let $r > 1$ and consider a system with $r$ equations. If it is irreducible, then we can choose $\ell = 1$ and 
	$\mathcal I_1 = \{ 1, \ldots, r \}$. Otherwise, it
	can be split (by definition) into two systems with indices in two disjoint nonempty index sets
	$\mathcal I$ and $\mathcal J$, respectively, such that $\mathcal I \cup \mathcal J = \{ 1, \ldots, r \}$. The systems
	 $\mS(\cI)$ and  $\mS(\mathcal J)$ have strictly less
	than $r$ equations, and therefore, relying on the inductive hypothesis, they can be split further into irreducible
	subsystem using the partitions \[
	  \mathcal I = \mathcal I_1 \cup \ldots \cup \mathcal I_{\ell_1}, \qquad 
	  \mathcal J = \mathcal J_1 \cup \ldots \cup \mathcal J_{\ell_2}.
	\]
	Then, 
	$\{ 1, \ldots, r \} = \mathcal I_1 \cup \ldots \cup \mathcal I_{\ell_1}
	\cup \mathcal J_1 \cup \ldots \cup \mathcal J_{\ell_2}$ yields a decomposition
	into irreducible systems with $\ell := \ell_1 + \ell_2$ components, 
	and this concludes the proof. 
	\end{proof}

Proposition \ref{prop:partition} shows that every system can be split into irreducible systems. 
To determine if a system is nonsingular, it is sufficient to answer the same question for its irreducible components, as stated in the following result.

	\begin{proposition}\label{thm:uniqirre}
	Let $\mS$ be the  
system \eqref{eq:gensystem} with $s$ matrix unknowns, 
and let $\mathcal I_1\cup\cdots\cup\mathcal I_\ell$ be a partition of $\{1,\ldots,r\}$ such that each system $\mS(\cI_j)$ is irreducible, for $j=1,\ldots,\ell$. The system $\mS$ is nonsingular if and only if the system $\mS(\cI_j)$ is nonsingular, for each $j=1,\ldots,\ell$.
	\end{proposition}
	\begin{proof}
	We shall show directly that $\mS$ has a unique solution if and only if $\mS(\cI_j)$ has a unique solution for each $j=1,\hdots,\ell$.
Any solution of $\mS$ yields a solution of $\mS(\cI_j)$, for each $j=1,\hdots,\ell$, and viceversa. Let us assume that $\mS$ has two different solutions $(X_1,\ldots,X_s)$ and $(Y_1,\ldots,Y_s)$. Then there exists some $1\leq p\leq s$ such that $X_p\ne Y_p$. If $p\in\cI_q$, for some $1\leq q\leq \ell$, then $\mS(\cI_q)$ has two different solutions, the first one containing $X_p$ and the second one containing $Y_p$. Conversely, if not every system $\mS(\cI_j)$ is nonsingular, then there is some $1\leq q\leq\ell$ such that either $\mS(\cI_q)$ is not consistent or it has two different solutions. In the first case, the whole system $\mS$ would not be consistent either. If $\mS(\cI_q)$ has two different solutions, $(X_1,\hdots,X_{s_q})$ and $(Y_1,\hdots,Y_{s_q})$, and $\mS(\cI_j)$ is consistent, for any $j\neq q$, then we can construct two different solutions of~$\mS$ by completing with $(X_1,\hdots,X_{s_q})$ and $(Y_1,\hdots,Y_{s_q})$, respectively, a solution of the remaining $\mS(\cI_j)$ for $j\neq q$.
	\end{proof}

Finally, we show that for nonsyngular systems, the number of equations and unknowns in each irreducible subsystem is the same.

\begin{proposition} \label{thm:equalirre}
	Let $\mS$ be the system \eqref{eq:gensystem} with $r$ matrix unknowns with size $n\times n$ and let $\mathcal I_1\cup\cdots\cup\mathcal I_\ell$ be a partition of $\{1,\ldots,r\}$ such that each system $\mS(\cI_j)$ is irreducible, for $j=1,\ldots,\ell$. Let $r_j$ and $s_j$ be the number of matrix equations and unknowns, respectively, of $\mS(\cI_j)$. If the system $\mS$ has a unique solution then $r_j=s_j$, for $j=1,\ldots,\ell$.
\end{proposition}
\begin{proof} If an  
irreducible system with $\widehat r$ equations and $\widehat s$ unknowns has a unique solution, then $\widehat s\le\widehat r$, since otherwise this system, considered as a linear system on the entries of the matrix unknowns, would have more unknowns than equations.

	Now, by contradiction, assume that $r_j\ne s_j$, for some $1\le j\le \ell$. Then, since $\sum_{j=1}^\ell r_j=\sum_{j=1}^\ell s_j=r$, there exists some $1\le p\le \ell$ such that $r_p<s_p$. Thus the system $\mS(\cI_p)$ cannot have a unique solution, and this contradicts Proposition~\ref{thm:uniqirre}.
	\end{proof}
	
The previous results show that, in order to analyze the nonsingularity of a system of $r$ matrix equations in $r$ matrix unknowns, we may assume that the system is irreducible. 

Moreover, Proposition \ref{thm:uniqirre} shows that a first step to compute the unique solution of a system of type \eqref{eq:gensystem} consists in splitting the system into irreducible systems and solving them separately.

%%%%%%%%%%%%%%%%%%%%%%%%%%%%%%%%
%%%%%%%%%%%%%%%%%%%%%%%%%%%%%%%%
%%%%%%%%%%%%%%%%%%%%%%%%%%%%%%%%
%%%%%%%%%%%%%%%%%%%%%%%%%%%%%%%%
%%%%%%%%%%%%%%%%%%%%%%%%%%%%%%%%

\subsection{Reduction to a system where every unknown appears twice}\label{sec:redtotwo}

We consider a nonsingular irreducible system of Sylvester-like equations and we want to prove that the system can be reduced to another one in which each unknown appears exactly twice (and in different equations, when the system has at least two equations). For this purpose, we need the following result.

\begin{theorem}\label{thm:justone} Let $\mS$ be an irreducible system of equations in the form \eqref{eq:gensystem} with $r>1$ equations and unknowns. If the unknown $X_{\alpha_k}$ appears in just one equation, say $A_kX_{\alpha_k}^{s_k}B_k-C_kX_{\beta_k}^{t_k}D_k=E_k$, then $\mS$ is nonsingular if and only if $A_k$ and $B_k$ are invertible and the system $\widetilde{\mS}$ formed by the remaining $r-1$ equations is nonsingular. Moreover $\wt \mS$ is irreducible.
\end{theorem}
\begin{proof} Note, first, that $\beta_k\neq \alpha_k$, and that the variable $X_{\beta_k}$ appears again in $\widetilde{\mS}$, otherwise $\mS$ would be reducible.
  Suppose first that $\widetilde{\mS}$ is nonsingular and $A_k,B_k$ are invertible. Then, the unique solution of $\mS$ is obtained by first solving $\widetilde{\mS}$ to get the value of all the variables except $X_{\alpha_k}$, and then computing $X_{\alpha_k}$ from
	\begin{equation}\label{eq:alfak}
		X_{\alpha_k}^{s_k}=A_k^{-1}(C_kX_{\beta_k}^{t_k}D_k+E_k)B_k^{-1}.
	\end{equation}
	If $\widetilde{\mS}$ has more than one solution, for $A_k$ and $B_k$ invertible, then \eqref{eq:alfak} produces multiple solutions to $\mS$. If $\widetilde{\mS}$ has no solution, then clearly $\mS$ has no solution either. If $A_k$ is singular, let $v$ be a nonzero vector such that $A_kv=0$; then, given any solution to~\eqref{eq:gensystem} we can replace $X_{\alpha_k}^{s_k}$ with $X_{\alpha_k}^{s_k} + vu^\top$, for any $u\in\mathbb{C}^n$, obtaining a new solution of~\eqref{eq:gensystem}, so $\mS$ does not have a unique solution. A similar argument can be used if $B_k$ is singular.

Moreover, $\widetilde{\mS}$ is irreducible. Otherwise, it could be split in two systems with different unknowns, and just one of them would contain $X_{\beta_k}$; adding the $k$th equation to this last system would give a partition of the original system $\mS$ in two systems with different unknowns.
\end{proof}

The proof of Theorem \ref{thm:justone} shows that, if an irreducible nonsingular system $\mS$ having $r>1$ unknowns contains an unknown appearing just once in $\mS$, then we can remove this unknown, together with its corresponding equation, to get a new irreducible system with $r-1$ equations and $r-1$ unknowns. Notice that the new system may have unknowns appearing just once, that can be removed if $r>2$, using Theorem \ref{thm:justone} again.

This elimination procedure can be repeated as long as the number of equations is greater than one and there is an unknown appearing just once. After a finite number of reductions (using Theorem \ref{thm:justone} repeatedly), we arrive at an irreducible system $\wt{\mathbb S}$, which has the same number $\wt r$ of equations and unknowns and either $\wt r=1$ or no unknown appears in just one equation. In both cases, all unknowns in $\wt{\mathbb S}$ appear just twice.
Moreover, 
$\widetilde\mS$ is nonsingular. Therefore, we can focus, from now on, on irreducible systems with the same number of equations and unknowns, and where each unknown appears exactly twice.

\subsection{Reduction to a periodic system with at most one $\star$}

In Section~\ref{sec:redtotwo} we have proved that, without loss of generality, and regarding nonsingularity, we can consider irreducible systems of $r$ Sylvester-like equations with $r$ matrix unknowns, any of which appearing just twice. Now, we want to show that from any system of the latter form, we can get an equivalent periodic system of the form \eqref{eq:periodic}.

We first note that, by renaming the unknowns if necessary, under these assumptions the system \eqref{eq:gensystem} can be written as
	\begin{equation} \label{eq:generalized_sylvester}
	\left\{\begin{array}{cccc}
			A_k X_k^{s_{k}} B_k - C_k X_{k+1}^{t_{k}} D_k &=& E_k,& k = 1, \ldots, r-1,\\
			A_r X_r^{s_{r}} B_r - C_r X_1^{t_{r}} D_r& =& E_r,\\
			\end{array}\right.
	\end{equation}
 where $s_k, t_k\in\{1,\star\}$. A way to show this is as follows. Let us start with $X_1$ and choose one of the two equations containing this unknown (there are at least two as long as the system contains at least two equations). Let this equation, with appropriate relabeling of the coefficients if needed, be $A_1X_1^{s_1}B_1-C_1X_{\alpha_1}^{t_1}D_1=E_1$. Now we look for the other equation containing $X_{\alpha_1}$. With a relabeling of the coefficients if needed, this equation is $A_2X_{\alpha_1}^{s_2}B_2-C_2X_{\alpha_2}^{t_2}D_2=E_2$, and we proceed in this way with $X_{\alpha_2}$ and so on with the remaining unknowns. Note that, during this process, it cannot happen that $\alpha_i=\alpha_j$ for $i\neq j$, since otherwise $X_{\alpha_i}$ would appear more than twice in the system. Therefore, at some point we end up with $\alpha_t=1$. If there were some $1\leq j\leq r$ such that $j\neq\alpha_i$, for all $i=1,\hdots,t$, then the system would be reducible. Hence, it must be $t=r$ and, by relabeling the unknowns as $\alpha_k=k+1$, for $k=1,\hdots,r-1$, and $\alpha_r=1$, we get the system in the form \eqref{eq:generalized_sylvester}.
 
 We now show that each periodic irreducible system of the form~\eqref{eq:generalized_sylvester} can be reduced to the simpler form \eqref{eq:periodic}, with at most one $\star$. This can be obtained by applying a sequence of $\star$ operations and renaming of variables, without further linear algebraic manipulations. This is stated in the following result.
	
	\begin{lemma}
		\label{lem:reduction-only-one-star}
		Given the system of generalized $\star$-Sylvester equations~\eqref{eq:generalized_sylvester}, there exists a system of the type
		\begin{equation}\label{eq:periodic2}
			\left\{\begin{array}{cccc}
			\wt A_k Y_k \wt B_k - \wt C_k Y_{k+1} \wt D_k &= &\wt E_k, & k = 1, \ldots, r-1, \\
			\wt A_r Y_r \wt B_r - \wt C_r Y_{1}^s \wt D_r &=& \wt E_r,
			\end{array}\right.
		\end{equation}
 		with $s\in\{1,\star\}$, and $u_k\in\{1,\star\}$, for $k=1,\ldots,r$, such that $Y_1,\ldots,Y_r$ is a solution of~\eqref{eq:periodic2} if and only if $X_1,\ldots,X_r$, with $X_k = Y_k^{u_k}$, is a solution of~\eqref{eq:generalized_sylvester}.
 		
	Moreover, $s=1$ if the number of $\star$ symbols appearing among $s_i,t_i$ in the original system~\eqref{eq:generalized_sylvester} is even, and $s=\star$ if it is odd.
	\end{lemma}
\begin{proof}	
		The proof of this result is constructive, i.e., it is presented in an algorithmic way that produces the system \eqref{eq:periodic2} from \eqref{eq:generalized_sylvester} by a sequence of transpositions and substitutions of the type $Y_k=X_k^{u_k}$, from which the statement follows.	
			
		The procedure has $r$ steps. At the first step we consider the first equation. If $s_1=\star$ then we apply the $\star$ operator to both sides of the equation, obtaining a new equivalent equation with no star on the first unknown:
		\[
			A_1 X_1^{\star} B_1 - C_1 X_2^{t_1} D_1 = E_2 \iff 
			B_1^\star X_1 A_1^\star - D_1^\star (X_2^{t_1})^{\star} C_1^\star = E_1^\star. 
		\]
We set $Y_1=X_1$ and $(\wt A_1,\wt B_1,\wt C_1,\wt D_1,\wt E_1)=(B_1^{\star},A_1^{\star},
			D_1^{\star},C_1^{\star},E_1^{\star})$. If $s_1=1$, then we set $Y_1=X_1$ as well and $(\wt A_1,\wt B_1,\wt C_1,\wt D_1,\wt E_1)=(A_1,B_1,C_1,D_1,E_1)$. In both cases, $u_1=1$ and the first equation has been replaced by $\wt A_1 Y_1\wt B_1-\wt C_1(X_2^{t_1})^{s_1}\wt D_1=\wt E_1$. Notice that, for $r=1$, we get an equivalent periodic system of the type \eqref{eq:periodic2} and then we are done.

If $r>1$, then we continue the first step of the procedure and check the second unknown of the first equation, namely $(X_2^{t_1})^{s_1}$, that can be $X_2$ or $X_2^{\star}$. If the second unknown is $X_2$, then we set $Y_2=X_2$ and $u_2=1$, otherwise we set $Y_2=X_2^\star$ and $u_2=\star$. In both cases we get an equation of the type
$\wt A_1 Y_1\wt B_1-\wt C_1Y_2\wt D_1=\wt E_1$, with no $\star$ in the unknowns. Replacing $X_2$ by $Y_2^{u_2}$ also in the second equation we get a system equivalent to \eqref{eq:generalized_sylvester} but with no $\star$ in the first equation.

The procedure can be repeated for the remaining equations. The second step works on the second equation, that now is of the form $A_2(Y_2^{u_2})^{s_2}B_2-C_2X_3^{t_2}D_2=E_2$. If $(Y_2^{u_2})^{s_2} = X_2$, then we can take $(\wt A_2, \wt B_2, \wt C_2, \wt D_2, \wt E_2) = (A_2, B_2, C_2, D_2, E_2)$; otherwise, $(Y_2^{u_2})^{s_2} = X_2^\star$, so we apply the operator $\star$ to the second equation, obtaining an equivalent one, and hence set $(\wt A_2, \wt B_2, \wt C_2, \wt D_2, \wt E_2) = (B_2^\star, A_2^\star, D_2^\star, C_2^\star, E_2^\star)$. Then we check if the other unknown appearing in the resulting equation is $X_3$ or $X_3^\star$, and proceed analogously. After $r-1$ steps we arrive at the last equation, which is of the form $A_rX_r^{s_r}B_r-C_rX_1^{t_s}D_r=E_r$, with $X_1=Y_1$ and either $X_r=Y_r$ or $X_r=Y_r^\star$. Therefore, there are four possible cases
\begin{eqnarray}
A_rY_rB_r-C_rY_1D_r&=&E_r,\label{case1}\\
A_rY_rB_r-C_rY_1^\star D_r&=&E_r,\label{case2}\\
A_rY_r^\star B_r-C_rY_1D_r&=&E_r,\label{case3}\\
A_rY_r^\star B_r-C_rY_1^\star D_r&=&E_r.\label{case4}
\end{eqnarray}
Cases \eqref{case1} and \eqref{case2} are already in the form required in \eqref{eq:periodic2}. For case \eqref{case3} we apply the $\star$ operator to this equation and arrive at
\[
\widetilde A_r Y_r \widetilde B_r-\widetilde C_r Y_1^\star \wt D_r=\widetilde E_r,
\]
with $(\wt A_r, \wt B_r, \wt C_r, \wt D_r, \wt E_r) = (B_r^\star, A_r^\star, D_r^\star, C_r^\star, E_r^\star)$, and in case \eqref{case4} we apply again the $\star$ operator to this equation and we get
\[
\widetilde A_r Y_r \widetilde B_r-\widetilde C_r Y_1 \wt D_r=\widetilde E_r,
\]
with $(\wt A_r, \wt B_r, \wt C_r, \wt D_r, \wt E_r) = (B_r^\star, A_r^\star, D_r^\star, C_r^\star, E_r^\star)$, as above. Therefore, in all cases we arrive at a system \eqref{eq:periodic2}.

Each of the transformations performed by the algorithm preserves the parity of the number of $\star$ symbols appearing within the equations, since each change of variables may swap the exponent, from $\star$ to $1$ or vice versa, in the two appearances of each unknown. Therefore, the second part of the statement follows.
\end{proof}

\begin{algorithm}
	\caption{Transformation of a periodic system into a system with just one $\star$. Vectors $s$ and $t$ contain the transpositions in the
		original system. The procedure
	returns the new coefficients, the vector $u$ so that $Y_k = X_k^{u_k}$, 
	and the symbol $t_r$ on $X_{r+1} = X_1$ 
	in the last equation (which is the only 
	entry in both $s$ and $t$ that could be a $\star$ after the procedure).}
	\begin{algorithmic}[1]
		\Procedure{GenerateSystem}{$A_k, B_k, C_k, D_k, E_k,s,t$}
			\State{$u_1\gets 1$}\Comment{$u_1$ is always $1$, 
			  since $Y_1 = X_1$}
		\For{$k = 1,\ldots,r$}
			\If{$s_k=1$}
			\State {$(\wt A_k,\wt B_k,\wt C_k,\wt D_k,\wt E_k)		
			\gets (A_k,B_k,C_k,D_k,E_k)$}
			\Else
			\State {$(\wt A_k,\wt B_k,\wt C_k,\wt D_k,\wt E_k)
			\gets (B_k^{\star},A_k^{\star},D_k^{\star},C_k^{\star},E_k^{\star})$}
			\State {Swap $t_k$}\Comment{Swap the value of $t_k$ between $1$ and $\star$}
			\EndIf
			\If{$k<r$}
			\State {$u_{k+1}\gets t_k$}
			\If{$t_k=\star$}
			\State {Swap $s_{k+1}$}\Comment{Swap the value of $s_{k+1}$ between $1$ and $\star$}
			\EndIf
			\EndIf
		\EndFor
		\State \Return $\wt A_k,\wt B_k,\wt C_k,\wt D_k,\wt E_k,u,t_r$
		\EndProcedure
	\end{algorithmic}
	\label{alg:1}
\end{algorithm}	

	The above results show that we can reduce the problem on the nonsingularity of \eqref{eq:gensystem} either to the problem of the nonsingularity of a periodic system
	of $r$ generalized Sylvester equations or to the problem of the nonsingularity of a periodic system of $r-1$ generalized Sylvester
	and one generalized $\star$-Sylvester equation. 

%%%%%%%%%%%%%%%%%%%%%%%%%%%%%%%%
%%%%%%%%%%%%%%%%%%%%%%%%%%%%%%%%
%%%%%%%%%%%%%%%%%%%%%%%%%%%%%%%%
%%%%%%%%%%%%%%%%%%%%%%%%%%%%%%%%
%%%%%%%%%%%%%%%%%%%%%%%%%%%%%%%%

\section{Reduction to a block triangular linear system}\label{sec:red2}

In Section~\ref{sec:red1} we have seen how a nonsingular system of general type \eqref{eq:gensystem} can be reduced to one or more independent periodic systems of the type \eqref{eq:periodic}, where all equations are generalized Sylvester equations except the last one, that is either a generalized Sylvester or a generalized $\star$-Sylvester equation.

Here we focus on a periodic system of type \eqref{eq:periodic}. First, we show in Section~\ref{sec:triangular} that it can be transformed into an equivalent periodic system with triangular coefficients. Then, in Section~\ref{sec:backsubstitution} we show that, in the cases $s=1$ and $s=\top$, the latter system is a linear system whose coefficient matrix is block triangular with diagonal blocks of order $r$ or $2r$. Finally, in Section~\ref{sec:linearizing} we show that the case $s=\ctop$ can be reduced to the case $s=1$.

The reduction to a special linear system allows one to deduce useful conditions for the nonsingularity of a system of generalized Sylvester equations and, moreover, to design an efficient numerical algorithm for its solution.

%%%%%%%%%%%%%%%%%%%%%%%%%%%%%%%%
%%%%%%%%%%%%%%%%%%%%%%%%%%%%%%%%
%%%%%%%%%%%%%%%%%%%%%%%%%%%%%%%%
%%%%%%%%%%%%%%%%%%%%%%%%%%%%%%%%
%%%%%%%%%%%%%%%%%%%%%%%%%%%%%%%%

	\subsection{Reduction to a system with triangular coefficients} \label{sec:triangular}

	We can multiply by suitable unitary matrices and perform a change of variables on the system \eqref{eq:periodic} which simultaneously make the matrices $A_k, B_k, C_k, D_k$ upper or lower (quasi-)triangular.
	
	\begin{lemma}\label{thm:lemtriang}
		There exists a change of variables of the form $\widehat{X}_k = Z_k^\ctop X_k \widehat{Z}_k$, with $Z_k,\widehat{Z}_k\in\mathbb{C}^{n\times n}$ unitary, for $k=1,2,\dots,r$, which simultaneously makes the coefficients $A_k,C_k$ of \eqref{eq:periodic} upper triangular, and the coefficients $B_k,D_k$ lower triangular, after pre-multiplying and post-multiplying the $k$th equation by appropriate unitary matrices $Q_k$ and $\wh Q_k$, respectively.
	\end{lemma}
	\begin{proof} 
		We distinguish the cases $s = 1$ and
		$s \in \{ \top, \ctop\}$. For both cases, we provide an appropriate
		change of variables to take the system in upper/lower triangular form, 
		based on the periodic Schur form of certain formal matrix products  (see Section~\ref{sec:prel}). 
		\begin{description}
			\item[Case $s = 1$] is already treated in \cite{byers-rhee}; we
			   report it here for completeness. Let 
			   \[
			   Q_k^\ctop A_k Z_k = \widehat{A}_k, \quad Q_k^\ctop C_k Z_{k+1} = \widehat{C}_k,\quad Z_{r+1}=Z_1,\quad k=1,2,\ldots,r,
			   \]
			   with $\widehat{A}_k,\widehat{C}_k$ upper triangular, be a periodic Schur form of 
			   $C_r^{-1}A_rC_{r-1}^{-1}A_{r-1}\dotsm C_1^{-1}A_1$, and 
			  \[
			  \widehat{Q}_k^\ctop B^\ctop_k \widehat{Z}_k = \widehat{B}^\ctop_k, \quad \widehat{Q}_k^\ctop D_k^\ctop \widehat{Z}_{k+1} = \widehat{D}_k^\ctop,\qquad Z_{r+1}=Z_r,\qquad k=1,2,\ldots,r,
			  \]
			  with $\widehat{B}_k^\ctop,\widehat{D}_k^\ctop$ upper triangular, be a periodic Schur form of 
			  $D_r^{-\ctop}B_r^\ctop D_{r-1}^{-\ctop}B_{r-1}^\ctop\dotsm D_1^{-\ctop}B_1^\ctop$. Setting $\widehat{X}_k = Z_k^\ctop X_k \widehat{Z}_k$ and
			  multiplying the equations in \eqref{eq:periodic} by $Q_k^\ctop$ from the left, 
			  and by $\widehat {Q}_k$ from the right yields a transformed 
			  system of equations with unknowns $\widehat X_k$ and 
			  upper/lower triangular coefficients, as claimed. 
		\item[Case $s \in \{\ctop,\top\}$] 
		  can be handled by considering the
		periodic Schur form \begin{align*}
		Q_k^\ctop A_k Z_k &= \widehat{A}_k, & Q_k^\ctop C_k Z_{k+1} &= \widehat{C}_k, & Z_{2r+1}&=Z_1,\\
		Q_{r+k}^\ctop B^s_k Z_{r+k} &= \widehat{B}^s_k, & Q_{r+k}^\ctop D_k^s Z_{r+k+1} &= \widehat{D}_k^s, & k&=1,2,\dots,r,
		\end{align*}
		of 
		$
		D_r^{-s}B_r^s D_{r-1}^{-s}B_{r-1}^s\dotsm D_1^{-s}B_1^s C_r^{-1}A_rC_{r-1}^{-1}A_{r-1}\dotsm C_1^{-1}A_1.
		$
		
		Performing the change of variables $\widehat X_k = Z_k^\ctop X_k (Z_{r+k}^s)^\ctop$ 
		and multiplying the equations in \eqref{eq:periodic} by $Q_k$ on the left and by
		$(Q_{r+k}^s)^\ctop$ on the right
		yields a system with upper/lower triangular coefficients
		in the unknowns $\widehat{X}_k$. Note that, for any matrix $M$, 
		$(M^s)^\ctop$ is equal to $M$ if $s = \ctop$ and $\overline M$ (the
		complex conjugate) if $s = \top$. 
		\end{description}
	\end{proof}
	
%%%%%%%%%%%%%%%%%%%%%%%%%%%%%%%%
%%%%%%%%%%%%%%%%%%%%%%%%%%%%%%%%
%%%%%%%%%%%%%%%%%%%%%%%%%%%%%%%%
%%%%%%%%%%%%%%%%%%%%%%%%%%%%%%%%
%%%%%%%%%%%%%%%%%%%%%%%%%%%%%%%%
%%%%%%%%%%%%%%%%%%%%%%%%%%%%%%%%
%%%%%%%%%%%%%%%%%%%%%%%%%%%%%%%%
%%%%%%%%%%%%%%%%%%%%%%%%%%%%%%%%
%%%%%%%%%%%%%%%%%%%%%%%%%%%%%%%%
%%%%%%%%%%%%%%%%%%%%%%%%%%%%%%%%

	\subsection{Reduction to a block upper triangular linear system for $s=1, \top$} \label{sec:backsubstitution}
	
A system like \eqref{eq:periodic} can be seen as a system of $n^2r$ equations in $n^2r$ unknowns in terms of the entries of the unknown matrices. This is a linear system for $s=1$ or $s=\top$, while in the case $s=\ctop $ it is not linear over $\C$ due to the conjugation. Nevertheless, it can be either transformed into a linear system over $\R$, by splitting the real and imaginary parts of both the coefficients and the unknowns (see Section~\ref{sec:alg-backsubstitution}), or into a linear system over $\C$ by doubling the size (see Section~\ref{sec:linearizing}).

A standard approach to get explicitly the matrix coefficient of the (linear) system associated with a system of Sylvester-like equations is to exploit the relation  $\vecop(A X B) = (B^\top \otimes A) \vecop{X}$ \cite[Lemma 4.3.1]{hj-topics}
where the $\vecop(\cdot)$ operator maps a matrix into the vector obtained by stacking its columns one on top of the other, and $A\otimes B$ is the Kronecker product of $A$ and $B$, namely the block matrix with blocks of the type $[a_{ij}B]$ (see \cite[Ch. 4]{hj-topics}). 

	Relying on the reduction scheme that we have presented in
	Section~\ref{sec:triangular}, 
	we may assume that the coefficients $A_k, C_k$, and $B_k, D_k$, in~\eqref{eq:periodic}
are upper and lower triangular matrices, respectively.  In this case the matrix of the linear system obtained after applying the $\vecop(\cdot)$ operator has a nice structure;{ }
 indeed, performing appropriate row and column permutations to the matrix (in other words, choosing an appropriate ordering of the unknowns), in Section~\ref{sec:transposepermutation}, we get a block upper triangular coefficient matrix, with diagonal blocks of dimensions $r$ or $2r$.

In the case where $s=1$, a characterization for nonsingularity was obtained in \cite{byers-rhee} (see Theorem~\ref{thm:byersrhee}). The approach followed in that reference is similar to the one we follow here. 
	
	We first deal with the cases $s\in \{ 1, \top \}$, which are 
	both linear, and for which we can directly give conditions based
	on the matrix representing the linear system in the entries of the unknowns. This 
	is the aim of Section~\ref{sec:transposepermutation}. 
	The case $s = \ctop $ can be reduced to the case $s=1$ by using specific developments which are contained in Section~\ref{sec:linearizing}.
	
%%%%%%%%%%%%%%%%%%%%%%%%%%%%%%%%
%%%%%%%%%%%%%%%%%%%%%%%%%%%%%%%%
%%%%%%%%%%%%%%%%%%%%%%%%%%%%%%%%
%%%%%%%%%%%%%%%%%%%%%%%%%%%%%%%%
%%%%%%%%%%%%%%%%%%%%%%%%%%%%%%%%

\subsubsection{Making the matrix coefficient block triangular} \label{sec:transposepermutation}

We assume that $A_k,C_k$ are upper triangular and $B_k,D_k$ are lower triangular, for $k=1,\hdots,r$.

Using the relation $\vecop(A X B) = (B^\top \otimes A) \vecop{X}$ we can rewrite the system~\eqref{eq:periodic}, for the case $s = 1$, with $r>1$, as the linear system
	\begin{equation}\label{eq:1matrix}
		\begin{bmatrix}
			B_{1}^\top \otimes A_1 & - D_{1}^\top \otimes C_1 \\
			& \ddots  & \ddots\\
			&& B_{r-1}^\top \otimes A_{k-1} & - D_{r-1}^\top \otimes C_{r-1} \\
			-D_{r}^\top \otimes C_{r} &&& B_{r}^\top \otimes A_{r} \\
		\end{bmatrix} {\cal X}  ={\cal E},
	\end{equation} 
	where the empty block entries should be understood as zero blocks, and 
		\[
			{\cal X} := \begin{bmatrix}
						\vecop X_1 \\
						\vdots \\
						\vecop X_r
					\end{bmatrix}, \qquad 
			{\cal E} := \begin{bmatrix}
								\vecop E_1 \\
								\vdots \\
								\vecop E_r
							\end{bmatrix}.
		\]

	In the case $s = \top$, with $r>1$, we have, instead
	\begin{equation}\label{eq:Tmatrix}
		\begin{bmatrix}
			B_{1}^\top \otimes A_1 & - D_{1}^\top \otimes C_1 \\
			& \ddots & \ddots \\
			&& B_{r-1}^\top \otimes A_{k-1} & - D_{r-1}^\top \otimes C_{r-1} \\
			- (D_{r}^\top \otimes C_{r}) P_{n,n} &&& B_{r}^\top \otimes A_{r} \\
		\end{bmatrix} {\cal X}  ={\cal E},
	\end{equation}
	where $P_{a,b}$ denotes the \emph{commutation matrix}, i.e., the permutation
	matrix such that $P_{a,b} \vecop{X} = \vecop (X^\top)$ for each $X\in\mathbb{R}^{a\times b}$ \cite[Th. 4.3.8]{hj-topics}. 
	
	In the case $r=1$, the system is $(B_1^\top\otimes A_1-D_1^\top\otimes C_1){\cal X}  ={\cal E}$ for $s=1$ and   $(B_1^\top\otimes A_1-(D_1^\top\otimes C_1)P_{n,n}){\cal X}  ={\cal E}$ for $s=\top$.
	
	In the following,
	we index the components of $\cal X$ by means of the triple 
	$(i,j,k)$, that denotes the $(i,j)$ entry of $X_k$. 
	This is just a shorthand for the component $(k-1)n^2 + (j-1)n + i$ of $\cal X$. 
	Notice that each coordinate of any of the systems \eqref{eq:1matrix} and \eqref{eq:Tmatrix} can be obtained by multiplying 
	one of the $r$ equations of \eqref{eq:periodic} by $e_i^\top$ on the
	left and by $e_j$ on the right, for appropriate $1\leq i,j\leq n$.
	
	We are interested in performing a permutation
	on systems \eqref{eq:1matrix} and \eqref{eq:Tmatrix} that takes them to block upper triangular form (independently
	on the presence of the permutation matrix $P_{n,n}$). The next
	Lemma shows that this is always possible. 
	 \begin{lemma} \label{lem:orderingtriangular} Let $A_k,C_k$ be $n\times n$ upper triangular matrices and $B_k,D_k$ be $n\times n$ lower triangular matrices, for $k=1,\hdots,r$.
 	Let $\mS$ be the system of $n^2r$ equations
	\begin{equation}\label{eq:eiej}
\left\{	 \begin{array}{cccc}e_i^\top (A_k X_k B_k - C_k X_{k+1} D_k) e_j& =& (E_k)_{ij}, & i,j = 1, \ldots, n, \ k = 1, \ldots, r-1,\\
	 e_i^\top(A_rX_rB_r-C_rX_1^sD_r)e_j&=&(E_r)_{ij}, & i,j = 1, \ldots, n, 
	 \end{array}\right.
	\end{equation}
	in the $n^2r$ unknowns $x_{ijk}$, for $i,j=1,\ldots,n$ and $k=1,\ldots,r$, where $x_{ijk}$ is the $(i,j)$ entry of $X_k$.
 	With a suitable ordering of the equations and unknowns, the coefficient matrix $M\in\C^{n^2r\times n^2r}$ of the system is block upper triangular, with diagonal blocks of size either $r \times r$ or $2r \times 2r$
	\end{lemma}
	\begin{proof}
	We define an ordering of the triples $(i,j,k)$ as follows. Define the following ordered sublists
	\begin{align*}
		\mathcal{L}_{ii}&= (i,i,1), (i,i,2), \ldots,(i,i,r), & 1&\leq i \leq n,\\
		\mathcal{L}_{ij}&= (i,j,1), (i,j,2), \ldots,(i,j,r),(j,i,1), (j,i,2), \ldots,(j,i,r), & 1&\leq j<i \leq n;
	\end{align*}
	then, we concatenate these sublists in lexicographic order of their index,
	\begin{equation}\label{eq:list}
		\mathcal{L}_{11},\mathcal{L}_{21},\mathcal{L}_{22},\mathcal{L}_{31},\mathcal{L}_{32},\mathcal{L}_{33},\mathcal{L}_{41},\mathcal{L}_{42},\mathcal{L}_{43},\mathcal{L}_{44},\dots,\mathcal{L}_{n1},\mathcal{L}_{n2},\ldots,\mathcal{L}_{nn}.
	\end{equation}
	In the matrix $M$, we sort the equations~\eqref{eq:eiej} (corresponding to rows) and the unknowns $x_{ijk}$ (corresponding to columns) according to this order~\eqref{eq:list} of the triples $(i,j,k)$. Grouping together the triples that belong to the same sublist $\mathcal{L}_{ij}$, we obtain a block partition of $M$ with $\frac{n(n+1)}{2}$ block rows and columns, each of size $r$ or $2r$, depending on whether $i \neq j$
	or $i = j$. 

	In order to simplify the notation, we set $x_{i,j,r+1}=x_{ij1}$ if $s=1$ and $x_{i,j,r+1}=x_{ji1}$ if $s=\star$. With this choice, $x_{ijk}$ and $x_{i,j,k'}$ belong to the same sublist ($\mathcal{L}_{ij}$ or $\mathcal{L}_{ji}$) for any $k,k'$, and whenever $i \leq \ell$ and $j \leq t$ the unknown $x_{ijk}$ belongs to a sublist that comes before $x_{\ell tk}$.

Since $A_k$ is upper triangular and $B_k$ is lower triangular, for a given $(i,j,k)$ we have
	\[
		(A_kX_kB_k)_{ij}=\sum_{\ell=1}^{n} (A_k)_{i\ell}\sum_{t=1}^n(X_k)_{\ell t}(B_k)_{t j}
		=\sum_{\ell=i}^n\sum_{t=j}^{n} (A_k)_{i\ell}(X_k)_{\ell t}(B_k)_{tj},
	\]
and similarly for $(C_kX_{k+1}D_k)_{ij}$. Thus the $(i,j,k)$ equation of the system is
	\[
	\sum_{\substack{i \leq \ell\\ j \leq t}} \bigl((A_k)_{i\ell}x_{\ell t k}(B_k)_{tj}-
	(C_k)_{i\ell} x_{\ell,t,k+1}(D_k)_{tj}\bigr)=
   (E_k)_{ij}.
	\]
	Hence an equation with index in $\mathcal{L}_{ij}$ contains only unknowns belonging to the sublist $\mathcal{L}_{ij}$ and to sublists that follow it in the order of \eqref{eq:list}. This proves that $M$ is block upper triangular.
\end{proof}

%%%%%%%%%%%%%%%%%%%%%%%%%%%%%%%%
%%%%%%%%%%%%%%%%%%%%%%%%%%%%%%%%
%%%%%%%%%%%%%%%%%%%%%%%%%%%%%%%%
%%%%%%%%%%%%%%%%%%%%%%%%%%%%%%%%
%%%%%%%%%%%%%%%%%%%%%%%%%%%%%%%%

	\subsubsection{Characterizing the diagonal blocks}
	\label{sec:char-diag-blocks}
	
	Both from the computational and from the theoretical point of view we
	are interested in characterizing the structure of the diagonal blocks
	of the coefficient matrix $M$ associated with the linear system obtained by applying the permutation of Lemma  \ref{lem:orderingtriangular}.
	
	Theoretically, this is interesting because the system \eqref{eq:periodic} is nonsingular if and only if
	the determinants of all diagonal blocks of $M$ are nonzero. 
	This will allow us to prove Theorems \ref{thm:1eigentheorem} and \ref{thm:stareigentheorem}.
		
	Computationally, this is relevant because these are the matrices
	that allow one to carry out the block
	back substitution process to compute the solution of \eqref{eq:periodic}, when it is unique.
	
	As already pointed out in Section~\ref{sec:transposepermutation}
	the diagonal blocks can be obtained by choosing a pair $(i,j)$ and selecting the equations 
	given by 
	\[
		\left\{\begin{array}{cccc}
			e_i^\top (A_k X_k B_k - C_k X_{k+1} D_k) e_j& = &(E_k)_{ij}, &    k = 1, \ldots, r-1, \\
			e_i^\top (A_r X_r B_r - C_r X_1^s D_r) e_j &=& (E_r)_{ij},
		\end{array}\right.
		 \]
	and the ones obtained by the pair $(j,i)$, and removing all the variables with indices
	different from $(i,j)$ and $(j,i)$. 
As mentioned in the proof of Lemma \ref{lem:orderingtriangular}, 
	these other variables have indices $(i',j',k')$ belonging to a subset $\mathcal{L}_{i',j'}$ that follows $\mathcal L_{ij}$ in the given order,
	and hence their value has already been computed in the back substitution process.
	When $i = j$ this gives us
	an $r \times r$ linear system,  
	otherwise we obtain a $2r \times 2r$ linear system. We denote them with $\mS_{ij}$, for $i\ge j$.
	
	Notice that this procedure can be carried out both in the case $s \in \{ 1, \top \}$ and
	in the $s = \ctop $ case, even if in the latter these systems are
	nonlinear. 
	
	\begin{lemma} \label{lem:determinant}
		Let $M$ be the following matrix:
		\[
			M = \begin{bmatrix}
				\alpha_1 &  \beta_1 \\
				& \ddots & \ddots \\
				& & \ddots & \beta_{p-1} \\
				\beta_p  &&& \alpha_p \\
			\end{bmatrix}.
		\]
		Then, $\det M =\displaystyle \prod_{k = 1}^p \alpha_k -(-1)^p \prod_{k = 1}^p \beta_k$. 
	\end{lemma}
	\begin{proof}
		Use Laplace's determinant expansion on the first column. 
	\end{proof}
	
	In the cases $s \in \{1, \top \}$, $\mS_{ii}$
	is an $r\times r$ linear system in the variables $(X_1)_{ii},\dots, (X_r)_{ii}$ 
	with coefficient matrix: 
	\begin{equation}\label{eq:sii}
	 M_{ii}:= \begin{bmatrix}
	(A_1)_{ii}(B_1)_{ii} & -(C_1)_{ii}(D_1)_{ii}\\
		 &  \ddots & \ddots \\
		 & & (A_{r-1})_{ii}(B_{r-1})_{ii} & -(C_{r-1})_{ii}(D_{r-1})_{ii}\\ 
		 -(C_r)_{ii}(D_r)_{ii}  &  &  &  (A_r)_{ii}(B_r)_{ii} \\
	 \end{bmatrix},
	\end{equation}
	for $r>1$ and $M_{ii}=(A_1)_{ii}(B_1)_{ii}-(C_1)_{ii}(D_1)_{ii}$ for $r=1$.
	
	According to Lemma~\ref{lem:determinant} we have: 
	\begin{equation}\label{eq:detmii}
		\det M_{ii} = \prod_{k = 1}^r (A_k)_{ii}(B_k)_{ii} - 
			\prod_{k = 1}^r (C_k)_{ii}(D_k)_{ii}\,.
	\end{equation}
A similar relation holds also when $i>j$ in the $s = 1$
	case, since $\mS_{ij}$ can be decoupled into two $r\times r$ systems. More precisely, in the case $s=1$, the coefficient matrix of $\mS_{ij}$ is block diagonal with two diagonal blocks, the top left block is
	\begin{equation}\label{eq:mijs1}
	M_{ij}:= \begin{bmatrix}
	(A_1)_{ii}(B_1)_{jj} & -(C_1)_{ii}(D_1)_{jj}\\
		 &  \ddots & \ddots \\
		 & & (A_{r-1})_{ii}(B_{r-1})_{jj} & -(C_{r-1})_{ii}(D_{r-1})_{jj}\\ 
		 -(C_r)_{ii}(D_r)_{jj}  &  &  &  (A_r)_{ii}(B_r)_{jj} \\
				\end{bmatrix},
	\end{equation}
	for $r>1$ and $M_{ij}=(A_1)_{ii}(B_1)_{jj}-(C_1)_{ii}(D_1)_{jj}$ for $r=1$,
	while the lower bottom block, $M_{ji}$, is obtained exchanging the roles of $i$ and $j$. From Lemma \ref{lem:determinant} we get:
	\begin{equation}\label{eq:detmij1}
	\det M_{ij}= \prod_{k = 1}^r (A_k)_{ii}(B_k)_{jj} - 
			\prod_{k = 1}^r (C_k)_{ii}(D_k)_{jj}.
	\end{equation}
	
	In the case $s = \top$, instead, 
	the systems $\mS_{ij}$ form a $2r\times 2r$ linear system in the variables $(X_k)_{ij}$, $(X_k)_{ji}$, for $k=1,\hdots,r$, with coefficient matrix
	\begin{equation}\label{eq:sij}
	M_{ij}:=\begin{bmatrix}
		\mathcal B_{ij} & -(C_r)_{ii}(D_r)_{jj} e_{r} e_1^\top \\
		-(C_1)_{jj}(D_1)_{ii} e_r e_1^\top & \mathcal B_{ji} \\
	\end{bmatrix},
	\end{equation}
	where 
	\[
	\mathcal B_{ij} = 
	\begin{bmatrix}
		(A_1)_{ii}(B_1)_{jj} & -(C_1)_{ii}(D_1)_{jj}\\
		& \ddots &\ddots\\
		& & \ddots & -(C_{r-1})_{ii}(D_{r-1})_{jj} \\
		& & & (A_r)_{ii}(B_r)_{jj} \\
	\end{bmatrix}.
	\]
	Thanks, again, to Lemma~\ref{lem:determinant}, this matrix has determinant equal to
	\begin{equation}\label{eq:detmij}
		\det M_{ij} = \prod_{k = 1}^r (A_k)_{ii}(B_k)_{ii}(A_k)_{jj}(B_k)_{jj} -\prod_{k = 1}^r (C_k)_{ii}(D_k)_{ii}(C_k)_{jj}
		(D_k)_{jj}. 
	 \end{equation} 

%%%%%%%%%%%%%%%%%%%%%%%%%%%%%%%%
%%%%%%%%%%%%%%%%%%%%%%%%%%%%%%%%
%%%%%%%%%%%%%%%%%%%%%%%%%%%%%%%%
%%%%%%%%%%%%%%%%%%%%%%%%%%%%%%%%
%%%%%%%%%%%%%%%%%%%%%%%%%%%%%%%%
 
	\subsection{Linearizing the case $s=\ctop $}\label{sec:linearizing}
 We have already mentioned that, when $s=\ctop $, the system~\eqref{eq:periodic} is not linear over the complex field, 
	since it involves not only the entries of the matrix $X_1$ but also their conjugates. A method to transform it into a linear system over $\C$ is as follows: in addition to the equations of the system, we consider the equations obtained by taking their conjugate transpose, namely
	\begin{equation*}
	\begin{array}{cccc}
		B_k^\ctop  X_k^\ctop  A_k^\ctop  - D_k^\ctop  X_{k+1}^\ctop  C_k^\ctop & =& E_k^\ctop ,&k=1,\hdots,r-1,\\
B_rX_r^\ctop A_r-D_r^\ctop X_1C_r^\ctop &=&E_r^\ctop .&
\end{array}
	\end{equation*}
If we consider $X_k$ and $X_k^\ctop $ as two separate variables, then this is a system of $2r$ generalized Sylvester equations in $2r$ matrix unknowns. We prove more formally that this process produces an equivalent system.
 
\begin{lemma} \label{complexlemma}
The system \eqref{eq:periodic}
is nonsingular if and only if the system
\begin{equation}\label{eq:s2}
\left\{\begin{array}{cccc}
 A_k X_k B_k -C_k X_{k+1} D_k &=& E_k, & k = 1, \ldots, r-1, \\
			A_r X_r B_r - C_r  X_{r+1}D_r& =& E_r,\\
			 B_k^\ctop  X_{r+k} A_k^\ctop -D_k^\ctop  X_{r+k+1} C_k^\ctop &=&E_k^\ctop ,&k=1,\ldots,r-1,\\
			 B_r^\ctop  X_{2r}A_r^\ctop -D_r^\ctop X_1C_r^\ctop &=&E_r^\ctop 
\end{array}\right.      
\end{equation}
is nonsingular.
			\end{lemma}
	\begin{proof}
	  We may consider only the case in which $E_k=0$: checking nonsingularity corresponds to checking that there are no solutions to this homogenous system apart from the trivial one $X_k=0$, for $k=1,\dots,r$.

		Let us first assume that \eqref{eq:periodic} has a nonzero solution $(X_1,\hdots,X_r)$. Then $(X_1,\hdots,X_r,X_1^\ctop ,\hdots,X_r^\ctop )$ is a nonzero solution of \eqref{eq:s2}. 
	
	Conversely, if $(X_1,\hdots,X_r,X_{r+1},\hdots,X_{2r})$ is a nonzero solution of \eqref{eq:s2}, then $(X_1+X_{r+1}^\ctop ,\hdots,X_r+X_{2r}^\ctop )$ is a solution of \eqref{eq:periodic}. If $(X_1+X_{r+1}^\ctop ,\hdots,X_r+X_{2r}^\ctop )=0$, then $X_{r+i}=-X_i^\ctop $, for $i=1,\hdots,r$, and then $\ci(X_1,\hdots,X_r)$ is a nonzero solution of \eqref{eq:periodic}.
	\end{proof}
	
	\begin{remark}  \label{remark:TH}
		The proof of Lemma~\ref{complexlemma} does not work if one replaces $\ctop $ with $\top$ everywhere: it breaks in the final part, because $\ci(X_1,\hdots,X_r)$ is not necessarily a solution of \eqref{eq:periodic} with $\star=\top$. Indeed, Lemma~\ref{complexlemma} is false with $\top$ instead of $\ctop $. Let us consider, for instance, the case $n=r=1$ and the equation $x_1 + x_1^\top = 2x_1 = 0$. This equation has only the trivial solution, but the linearized system
	\[
	\begin{cases}
		z_1 + z_2 = 0\\
		z_1 + z_2 = 0
	\end{cases}
	\]
	has infinitely many solutions.
	\end{remark}
	
	Another relevant difference between the $\star=\top$ and the $\star=\ctop $ cases is the following. System \eqref{eq:periodic} is nonsingular if and only if the system obtained after replacing the minus sign in the last equation by a plus sign
	\begin{equation}\label{eq:plusign}
	\left\{
\begin{array}{cccc}
A_kX_kB_k-C_kX_{k+1}D_k&=&E_k,& k=1,\ldots,r-1,\\
A_rX_rB_r+C_rX_1^\ctop D_r&=&E_r
\end{array}
\right.
	\end{equation} 
is nonsingular. To see this, reduce again to the case $E_k=0$ for all $k=1,\dots,r$ and note that if $(X_1,\hdots,X_r)$ is a nonzero solution of \eqref{eq:periodic} then $\ci(X_1,\hdots,X_r)$ is a nonzero solution of \eqref{eq:plusign}, and viceversa. This property no longer holds true with $s=\top$.

%%%%%%%%%%%%%%%%%%%%%%%%%%%%%%%%
%%%%%%%%%%%%%%%%%%%%%%%%%%%%%%%%
%%%%%%%%%%%%%%%%%%%%%%%%%%%%%%%%
%%%%%%%%%%%%%%%%%%%%%%%%%%%%%%%%
%%%%%%%%%%%%%%%%%%%%%%%%%%%%%%%%
%%%%%%%%%%%%%%%%%%%%%%%%%%%%%%%%
%%%%%%%%%%%%%%%%%%%%%%%%%%%%%%%%
%%%%%%%%%%%%%%%%%%%%%%%%%%%%%%%%
%%%%%%%%%%%%%%%%%%%%%%%%%%%%%%%%
%%%%%%%%%%%%%%%%%%%%%%%%%%%%%%%%

%%%%%%%%%%%%%%%%%%%%%%%%%%%%%%%%
%%%%%%%%%%%%%%%%%%%%%%%%%%%%%%%%
%%%%%%%%%%%%%%%%%%%%%%%%%%%%%%%%
%%%%%%%%%%%%%%%%%%%%%%%%%%%%%%%%
%%%%%%%%%%%%%%%%%%%%%%%%%%%%%%%%
%%%%%%%%%%%%%%%%%%%%%%%%%%%%%%%%
%%%%%%%%%%%%%%%%%%%%%%%%%%%%%%%%
%%%%%%%%%%%%%%%%%%%%%%%%%%%%%%%%
%%%%%%%%%%%%%%%%%%%%%%%%%%%%%%%%
%%%%%%%%%%%%%%%%%%%%%%%%%%%%%%%%

	\section{Proofs of the main results} \label{sec:existenceuniqueness}

	Here we prove Theorems~\ref{thm:1eigentheorem}--\ref{thm:stareigentheorem}, with the aid of all previous developments. We start with Theorem~\ref{thm:1eigentheorem}.
	
	\begin{proof}[Proof of Theorem~{\rm\ref{thm:1eigentheorem}}] 
	We can consider only the case in which $E_i=0$, $i=1,2,\dots,r$. Using the periodic Schur form of the formal products \eqref{eq:1formalproduct} we may consider the equivalent system
(see the proof of Lemma \ref{thm:lemtriang}) 
\[
\left\{
\begin{array}{cccc}
\wh A_k X_k\wh B_k-\wh C_kX_{k+1}\wh D_k&=&0,& k=1,\ldots,r-1,\\
\wh A_r X_r\wh B_r-\wh C_rX_1\wh D_r&=&0,
\end{array}
\right.
\]  
where, for each $k$, the matrices $\wh A_k$ and $\wh C_k$ are upper triangular and $\wh B_k$ and $\wh D_k$ are lower triangular. If the formal products \eqref{eq:1formalproduct} are regular, then their eigenvalues are the ratios $\lambda_i := \prod_{k=1}^r \frac{(\wh A_k)_{ii}}{(\wh C_k)_{ii}}$, 
		$\mu_i := \prod_{k=1}^r \frac{(\wh D_k)_{ii}}{(\wh B_k)_{ii}}$, respectively, for $i=1,\hdots,n$ (they are allowed to be $\infty$).
		
With this triangularity assumption, in Lemma \ref{lem:orderingtriangular} we have shown that the system of Sylvester equations is equivalent to a block upper triangular system whose matrix coefficient has determinant $\delta:=\prod_{i,j=1}^n \det(M_{ij})$, where $M_{ij}$ is defined in \eqref{eq:sii} and \eqref{eq:mijs1}.

In summary, the system of Sylvester equations is nonsingular if and only if $\delta\ne 0$, which, using \eqref{eq:detmii} and \eqref{eq:detmij1}, is equivalent to requiring
\begin{equation}\label{nothingcond}
\prod_{k=1}^r (\wh A_k)_{ii}(\wh B_k)_{jj} \neq \prod_{k=1}^r (\wh C_k)_{ii}(\wh D_k)_{jj},\qquad i,j=1,\ldots,n.
\end{equation}

If $\delta\ne 0$, then it cannot happen that $\prod_k(\wh A_k)_{ii}$ and $\prod_k(\wh C_k)_{ii}$ are both zero or that $\prod_k(\wh B_k)_{ii}$ and $\prod_k(\wh D_k)_{ii}$ are both zero and thus the formal products are regular. Moreover, condition \eqref{nothingcond} implies that $\la_i\ne \mu_j$ for any $i,j=1,\ldots,n$ and thus the two products have disjoint spectra.

On the contrary, if $\delta=0$ then the equality holds in \eqref{nothingcond} for some $i$ and $j$. One can check that this condition implies that either one of the two formal products is singular or $\la_i=\mu_j$ and they cannot have disjoint spectra.
	\end{proof}   
	
We now give the proof of Theorem~{\rm\ref{thm:stareigentheorem}}
separating the cases $\star=\top$ and $\star=\ctop $ since the techniques we use are different.
	
		 \begin{proof}[Proof of Theorem~{\rm\ref{thm:stareigentheorem}} for $\star=
		 \top$] 
Proceeding as in the proof of Theorem \ref{thm:1eigentheorem}, we use the periodic Schur form of the formal product \eqref{eq:magicproduct} to get the equivalent system
(see the proof of Lemma \ref{thm:lemtriang}) 
\[
\left\{
\begin{array}{cccc}
\wh A_k X_k\wh B_k-\wh C_kX_{k+1}\wh D_k&=&0,& k=1,\ldots,r-1,\\
\wh A_r X_r\wh B_r-\wh C_rX_1^\top\wh D_r&=&0,
\end{array}
\right.
\]  
where, for each $k$, the matrices $\wh A_k$ and $\wh C_k$ are upper triangular and $\wh B_k$ and $\wh D_k$ are lower triangular. If the formal product \eqref{eq:1formalproduct} is regular, then its eigenvalues are the ratios $\lambda_i := \prod_{k=1}^r \frac{(\wh A_k)_{ii}(\wh B_k)_{ii}}{(\wh C_k)_{ii}(\wh D_k)_{ii}}$, 
for $i=1,\hdots,n$.
		
With this triangularity assumption, in Lemma \ref{lem:orderingtriangular} we have shown that the previous system is equivalent to a block upper triangular system whose coefficient matrix has determinant $\delta:=\displaystyle\prod_{i=1}^n \det(M_{ii})\displaystyle\prod_{\substack{i,j=1 \\i<j}}^n\det(M_{ij})$, with $M_{ii}$ as in \eqref{eq:sii} and $M_{ij}$, for $i\ne j$, as in \eqref{eq:sij}.

In summary, the system of Sylvester-like equations is nonsingular if and only if $\delta\ne 0$, that, using \eqref{eq:detmii} and \eqref{eq:detmij}, is equivalent to requiring
			 \begin{equation}\label{eq:Tcond}
			 \begin{array}{ll}
				 \displaystyle\prod_{k=1}^r (\wh A_k)_{ii}(\wh B_k)_{ii} \neq \prod_{k=1}^r (\wh C_k)_{ii}(\wh D_k)_{ii}\,, & i =1,\dots,n,\\
				 \displaystyle \prod_{k=1}^r (\wh A_k)_{ii}(\wh B_k)_{ii}(\wh A_k)_{jj}(\wh B_k)_{jj} \neq \prod_{k=1}^r (\wh C_k)_{ii}(\wh D_k)_{ii}(\wh C_k)_{jj}(\wh D_k)_{jj}\,, & i \neq j.
				 \end{array}
\end{equation}
If $\delta\ne 0$, then it cannot happen that $\prod_k(\wh A_k)_{ii}(\wh B_k)_{ii}$ and $\prod_k(\wh C_k)_{ii}(\wh D_k)_{ii}$ are both zero, for some $i$, thus the formal product \eqref{eq:1formalproduct} is regular. Moreover, conditions \eqref{eq:Tcond} imply that 
\[
\left\{\begin{array}{l}
\la_i\ne 1,\qquad i=1,\ldots,n\\
\la_i\ne\la_j^{-1},\qquad i\ne j,
\end{array}\right.
\]
and this implies in turn that the spectrum $\Lambda(\Pi)\setminus\{-1\}$ is reciprocal free and the multiplicity of $\{-1\}$ is at most one.

On the contrary, if $\delta=0$ then the equality holds in \eqref{eq:Tcond} above for some $i$ or below for some pair $(i,j)$, with $i\ne j$. One can check that this condition implies that one of the following cases holds: (a) the formal product is singular; (b) $\la_i=1$, for some $i$, and thus $\Lambda(\Pi)\setminus\{-1\}$ is not reciprocal free; (c) $\la_i=1/\mu_j\ne -1$, for some $i\ne j$, and thus $\Lambda(\Pi)\setminus\{-1\}$ is not reciprocal free; (d) $\la_i=1/\mu_j=-1$ and the multiplicity of $-1$ is greater than $1$.
\end{proof}

	Using Lemma~\ref{complexlemma}, the following argument allows us to obtain Theorem~\ref{thm:stareigentheorem} with $\star=\ctop $ directly as a consequence of Theorem~\ref{thm:1eigentheorem}.

	\begin{proof}[Proof of Theorem~{\rm\ref{thm:stareigentheorem}} for $\star=\ctop $]
	Let us start from a system of the form~\eqref{eq:periodic} with $s=\ctop $.
	Lemma~\ref{complexlemma} shows that it is nonsingular if and only if the larger linear system \eqref{eq:s2} is nonsingular. System~\eqref{eq:s2} is a system of $2r$ generalized Sylvester equations with $s=1$. Hence we can apply Theorem~\ref{thm:1eigentheorem} to this system, obtaining that~\eqref{eq:s2} is nonsingular if and only if the two formal products
	\[
		\Pi_1 := \Pi = D_r^{-\ctop }B_r^\ctop D_{r-1}^{-\ctop }B_{r-1}^\ctop \dotsm D_1^{-\ctop }B_1^\ctop  C_r^{-1}A_r C_{r-1}^{-1}A_{r-1} \dotsm C_1^{-1}A_1
	\]
	and
	\[
		\Pi_2 := C_r^\ctop  A_r^{-\ctop } C_{r-1}^\ctop  A_{r-1}^{-\ctop } \dotsm C_1^\ctop  A_1^{-\ctop }D_r B_r^{-1} D_{r-1} B_{r-1}^{-1} \dotsm D_1 B_1^{-1} 
	\]
 are regular and have no common eigenvalues. If $\lambda_1,\lambda_2,\dots,\lambda_n$ denote the eigenvalues of $\Pi_1$, then the eigenvalues of the formal product
	\[
		\Pi_2^{-\ctop } := C_r^{-1} A_r C_{r-1}^{-1} A_{r-1} \dotsm C_1^{-1} A_1 D_r^{-\ctop } B_r^{\ctop } D_{r-1}^{-\ctop } B_{r-1}^{\ctop } \dotsm D_1^{-\ctop } B_1^{\ctop }
	\]
	are again $\lambda_1,\lambda_2,\dots,\lambda_n$, because $\Pi_2^{-\ctop }$ differs from $\Pi_1$ only by a cyclic permutation of the factors. This proves that the eigenvalues of $\Pi_2$ are $(\overline\lambda_1)^{-1},(\overline\lambda_2)^{-1},\allowbreak\hdots,(\overline\lambda_n)^{-1}$, so they are distinct from those of $\Pi_1$ if and only if $\Lambda(\Pi_1)$ is a $\ctop $-reciprocal free set. 
	\end{proof}

	This proof shows clearly the connection between the condition on a single formal product in Theorem~\ref{thm:1eigentheorem} and the condition on two products in Theorem~\ref{thm:stareigentheorem}. Unfortunately, we were unable to find a simple modification of this argument that works for the case $\star=\top$, mostly due to the issue presented in Remark~\ref{remark:TH}.

\section{An $O(n^3r)$ algorithm for computing the solution}
\label{sec:algorithm}

Here we describe an efficient algorithm for the solution of a nonsingular system of $r$ Sylvester-like equations~\eqref{eq:gensystem} of size $n\times n$.
We follow the big-oh notation $O(\cdot)$, as in \cite{higham-accuracy}, for both large and small quantities, and we use the number of floating point operations (flops) as a complexity measure.

The tools needed to develop the algorithm are the same used, in the previous sections, for the nonsingularity results.
In the description of the algorithm we focus on the complex case and
so we consider triangular coefficients. However, a solution with quasitriangular
forms in case of real data can be done following a similar procedure.

We proceed through the following steps:
\begin{enumerate}
\item\label{it:alg-reduction-periodic}(Step 1) We perform a suitable
  number of substitutions, changes and elimination of variables, in order to
  transform the system into irreducible systems of periodic
  form~(%
  \ref{eq:periodic}%
  ), as described in Section~\ref{sec:red1}.
\item\label{it:alg-reduction-triangular}(Step 2) For each (irreducible) periodic
  system, we compute a periodic Schur decomposition to reduce the
  coefficients, say $A_k, B_k, C_k, D_k$, to upper and lower
  triangular forms, as described in Section~\ref{sec:triangular}.
\item\label{it:alg-backsubstitution}(Step 3)
	Since the resulting systems can be seen as  essentially block triangular linear systems (as described in Section~\ref{sec:transposepermutation}),
	we solve them by back substitution.
\item\label{it:alg-remaining}(Step 4)
	We compute the value of the variables that have been eliminated in Step 1 (using Theorem~\ref{thm:justone}).
\end{enumerate}

This section describes how to handle these steps algorithmically. 
Moreover, we perform an analysis of
the computational costs, showing that the solution can be computed
in $O(n^3r)$ flops, and we prove a backward stability result for
the computed solution.

We discuss Step~\ref{it:alg-reduction-periodic} in Section~\ref{sec:algored}. 
Step~\ref{it:alg-reduction-triangular}
amounts to computing a periodic Schur factorization,
which can be carried out in $O(n^3 r)$ flops; we refer to \cite{bojanczyk1992periodic} for details
concerning it.

Step~\ref{it:alg-backsubstitution} is the one that requires more discussion; we devote Sections~\ref{sec:alg-backsubstitution}--\ref{sec:solvingB} to it. Moreover, we perform a backward error
analysis for the resulting algorithm in Section~\ref{sec:backward-error}.
We focus on the case $s=\star$, since the case $s=1$ can be found in~\cite{byers-rhee}.
The cases
$\star = \top$ and $\star = \ctop$ are handled in a similar way,
but the former is easier to describe since the associated system
is linear, without the need of separating the real
and imaginary parts. We describe accurately the procedure for $\star = \top$, and briefly explain the modifications
needed for $\star = \ctop$. The procedure for $r=1$ is the same as the one proposed in \cite{dd11}, and thus our algorithm can be seen as a generalization of the one presented in \cite{dd11}.

Finally, Step~\ref{it:alg-remaining} amounts to applying formula~\eqref{eq:alfak} several times.

%%%%%%%%%%%%%%%%%%%%%%%%%%%%%%%%
%%%%%%%%%%%%%%%%%%%%%%%%%%%%%%%%
%%%%%%%%%%%%%%%%%%%%%%%%%%%%%%%%
%%%%%%%%%%%%%%%%%%%%%%%%%%%%%%%%
%%%%%%%%%%%%%%%%%%%%%%%%%%%%%%%%

\subsection{An algorithm for the reduction step} \label{sec:algored}

We describe how Step~\ref{it:alg-reduction-periodic} can be implemented in $O(r)$ operations. This requires concepts and tools from graph theory, that can be found in \cite{cormen}. Technically, there are no floating-point operations, so one could argue that this step has cost $0$ in our model, but nevertheless it is useful to have an efficient way to perform it on a real-world computer.

Consider the undirected multigraph with self loops in which the nodes are the unknowns $X_1,\dots,X_r$, and there is an edge $(X_i,X_j)$ for each equation in which $X_i$ and $X_j$ appear. A self loop arises when an equation contains just one variable, and multiple edges arise when the same two unknowns appear in several equations. 

Reducing the system into irreducible subsystems corresponds to identifying the connected components of this graph, which can be done with $O(r)$ operations, since it has $r$ edges. We now consider each connected component $\mS(\cI_k)$ separately; if the system is irreducible, the corresponding subgraph $(V_k,\mathcal{E}_k)$ has $r_k$ nodes and $r_k$ edges (see Theorem~\ref{thm:equalirre}). 
Removing from $\mathcal{E}_k$ the self loops and the repeated edges (leaving just
one of them for each occurrence), we get a connected subgraph $(V_k,\widetilde{\mathcal{E}}_k)$. If $(V_k,\mathcal{E}_k)$ had two self loops or one self-loop and a multiple edge or two multiple edges or a multiple edge with more than two edges, then $(V_k,\widetilde{\mathcal{E}}_k)$ would be a connected graph with less than $r_k-1$
edges and $r_k$ nodes and this cannot happen.
Thus, there are three possible cases:
\begin{enumerate}
\item[] Case 1. $(V_k,\mathcal{E}_k)$ has no self loops and no multiple edges;
\item[] Case 2. $(V_k,\mathcal{E}_k)$ has one self loop and no multiple edges;
\item[] Case 3. $(V_k,\mathcal{E}_k)$ has no self loops and one double edge.
\end{enumerate}

After removing the self loop or the double edge (if any), choose an arbitrary node of the resulting graph $(V_k,\widetilde{\mathcal{E}}_k)$ as root, and perform a graph visit using breadth-first search (BFS,~\cite{cormen}). Since $(V_k,\widetilde{\mathcal{E}}_k)$ is connected, this visit will find all its vertices and form a predecessor subgraph $\mathcal{T}$ that contains $r_k-1$  edges of $(V_k,\widetilde{\mathcal{E}}_k)$~\cite{cormen}. In any of the three cases above, $\mathcal T$ is a tree obtained from $(V_k,\widetilde{\mathcal{E}}_k)$ removing one edge; let $(i,j)$ be this missing edge. 

The two nodes $i,j$ are connected by a path in $\mathcal{T}$ via their least common ancestor. In Case 1 this path  can be determined from the predecessor subgraph structure: for instance, build the paths from $i$ and $j$ to the root of $\mathcal{T}$ and remove their common final part; in Case 2, we have $i=j$ and the path is empty; in Case 3, the path is the edge in $\mathcal T$ connecting $i$ and $j$. Together with the removed edge $(i,j)$, this path forms a cycle $(\mathcal{C},\mathcal{E}_{\mathcal C})$ in $(V_k,\mathcal{E}_k)$. The graph $(V_k,\mathcal{E}_k\setminus \mathcal{E}_{\mathcal C})$ contains no cycles, because it is a subgraph of the predecessor subgraph. Moreover, in $(V_k,\mathcal{E}_k\setminus \mathcal{E}_{\mathcal C})$ each node is connected to exactly one node of the cycle $\mathcal{C}$ (because if it were connected to more than one, this would form a cycle in $\mathcal{T}$). Hence, $(V_k,\mathcal{E}_k \setminus \mathcal{E}_{\mathcal C})$ is a collection of trees, each containing exactly one node of $\mathcal{C}$. We perform a visit of each of these trees, starting from its unique node $c \in \mathcal{C}$. The variables corresponding to the nodes other than $c$ in this tree can be eliminated one by one, starting from the leaves (in the reverse of the order in which they are discovered by the BFS), with the elimination step described in Section~\ref{sec:redtotwo}, which removes a degree-1 tree from the graph. This elimination procedure reduces the system of equations associated to $V_k$ to the one associated to $\mathcal{C}$, which is a periodic system.

All the steps described above can easily be implemented with $O(r)$ operations---$O(r_k)$ for each connected component---just by operations on the indices. Once we have identified which cycles are formed, the coefficients can be swapped, transposed and conjugated as needed in $O(n^2r)$ operations (in-place, if one wishes to minimize the space overhead).

\subsection{Solving the triangular system %\eqref{eq:eiej}
}
\label{sec:alg-backsubstitution}

We consider the block-triangular system~\eqref{eq:eiej} with $s=\top$,
ordered according to~\eqref{eq:list}, as described in Lemma~\ref{lem:orderingtriangular}.  
This system is block upper
triangular with $\frac{n(n+1)}{2}$ diagonal blocks of order $r$ and
$2r$. We refer to the linear systems corresponding to these diagonal blocks as the
{\em small systems} $\mathbb S_{ij}$.

We provide in this section a high-level overview of the solution of this system by
block back substitution, and in Sections~\ref{sec:computingF} and
\ref{sec:solvingB} we describe how to perform it within the required
computational cost.

At each of the $\frac{n(n+1)}{2}$ steps of the back substitution
process, we need to solve a square linear system of the form:
\begin{equation}\label{eq:mxij}
	M_{ij} \mathcal X_{ij} = \mathcal{E}_{ij} - \mathcal{F}_{ij},
\end{equation}
with $M_{ij}$ as in \eqref{eq:sii} (when $i=j$) or \eqref{eq:sij} (when $i\neq j$); the vector $\mathcal X_{ij}$ has $r$ (if $i = j$) or $2r$ (if $i\neq j$)
components, obtained by stacking vertically all the entries $(X_1)_{ii},\ldots,(X_r)_{ii}$ (when $i=j$) or $(X_1)_{ij},\ldots,(X_r)_{ij}$ followed by $(X_1)_{ji},\ldots,(X_r)_{ji}$ (when $i\ne j$); the vector $\mathcal{F}_{ij}$ is defined as
$$
\mathcal{F}_{ij}:=\left\{\begin{array}{cc}w_{ii}&\mbox{if $i=j$},\\\left[ \begin{smallmatrix} w_{ij} \\ w_{ji} \end{smallmatrix} \right]&\mbox{otherwise},\end{array}\right.
$$
	where $w_{ij}$ is given by
\begin{equation} \label{eq:Fij}
		w_{ij} := \begin{bmatrix}
			v_{ij1} \\
			\vdots \\
			v_{ijr} \\
                      \end{bmatrix}, \ \ v_{ijk} :=
                      \sum_{\substack{s\ge i, t\ge j \\ (s,t)\neq(i,j)}}
                      \left(( A_k )_{is} ( X_k )_{st} ( B_k )_{tj} - (
                        C_k )_{is} ( X_{k+1} )_{st} ( D_k
                        )_{tj}\right);
\end{equation}
and $\mathcal{E}_{ij}$ contains all the entries in position $(i,j)$ (when $i=j$) or $(i,j)$ and $(j,i)$ (when
$i \neq j$) of $E_1, \ldots, E_r$ stacked vertically, according to the
order in $\mathcal{F}_{ij}$. We identify
$X_{r+1}$ with $X_1^\star$ for simplicity. 

Note that the values of the unknowns appearing in $\mathcal{F}_{ij}$
have been already computed if the linear systems are solved by block back substitution in the reverse of the order in~\eqref{eq:list}. 

The case $s = \ctop $ can be handled in a similar way, even if the
associated system $\mathbb S$ is nonlinear. In Section~\ref{sec:linearizing}, we have seen how the system can be linearized over $\C$ by doubling the number of equations. In order to use real arithmetic, here we follow a different approach: we consider
it as a larger linear system over $\R$ of double the dimension in
the variables $\re(\mathcal X_{ij})$ and $\im(\mathcal X_{ij})$.
More precisely, the system $\mathbb S_{ii}$, when $s = \ctop $, is equivalent
to the linear system over $\mathbb R$ defined, for $r>1$, by
\[
	\begin{bmatrix}
		\alpha_1 & \beta_1 \\
		& \ddots & \ddots  \\
		&& \ddots & \beta_{r-1} \\
		\beta_r &&& \alpha_r \\
	\end{bmatrix}
	\begin{bmatrix}
		Z_1 \\
		\vdots \\
		Z_{r-1} \\
		Z_r \\
	\end{bmatrix} =
	\begin{bmatrix}
		U_1 \\
		\vdots \\
		U_{r-1} \\
		U_r \\
	\end{bmatrix},
\qquad 
\left\{
\begin{array}{ll}
Z_k & =\begin{bmatrix}
\re(X_k)_{ii}\\
\im(X_k)_{ii}
\end{bmatrix},\\[2ex]
U_k & =\begin{bmatrix}
\re((E_k)_{ii}-(v_{ii})_{k})\\
\im((E_k)_{ii}-(v_{ii})_{k})
\end{bmatrix},
\end{array}
\right.
\]
where $\alpha_k, \beta_k$ are $2 \times 2$ matrices defined, respectively, by
\[
	\begin{bmatrix}
		\re((A_k)_{ii} (B_k)_{ii}) & - \im((A_k)_{ii} (B_k)_{ii}) \\
		\im( ( A_k )_{ii} ( B_k )_{ii} ) & \re( ( A_k)_{ii} ( B_k )_{ii} ) \\
	\end{bmatrix}, \quad
	-\begin{bmatrix}
		\re((C_k)_{ii} (D_k)_{ii}) & - \im((C_k)_{ii} (D_k)_{ii}) \\
		\im( ( C_k )_{ii} ( D_k )_{ii} ) & \re( ( C_k)_{ii} ( D_k)_{ii} ) \\
	\end{bmatrix}, 
\]
when $k < r$, and by 
\[
	\begin{bmatrix}
		\re((A_r)_{ii} (B_r)_{ii}) & - \im((A_r)_{ii} (B_r)_{ii}) \\
		\im( ( A_r )_{ii} ( B_r )_{ii} ) & \re( ( A_r)_{ii} ( B_r )_{ii} ) \\
	\end{bmatrix}, \quad
	-\begin{bmatrix}
		\re((C_r)_{ii} (D_r)_{ii}) & \im((C_r)_{ii} (D_r)_{ii}) \\
		\im( ( C_r )_{ii} ( D_r )_{ii} ) & - \re( ( C_r)_{ii} ( D_r )_{ii} ) \\
	\end{bmatrix}, 
\]
when $k = r$. Notice that the only differences between the two cases 
are the signs in the matrix on the right; this is due to
the conjugation appearing in the last equation. 

For $r=1$, the matrix coefficient is 
\[
	\begin{bmatrix}
		\re((A_1)_{ii} (B_1)_{ii}) & - \im((A_1)_{ii} (B_1)_{ii}) \\
		\im( ( A_1 )_{ii} ( B_1 )_{ii} ) & \re( ( A_1)_{ii} ( B_1 )_{ii} ) \\
	\end{bmatrix}-
	\begin{bmatrix}
		\re((C_1)_{ii} (D_1)_{ii}) & \im((C_1)_{ii} (D_1)_{ii}) \\
		\im( ( C_1 )_{ii} ( D_1 )_{ii} ) & - \re( ( C_1)_{ii} ( D_1)_{ii} ) \\
	\end{bmatrix}. 
\]

The systems obtained for
$\mathbb S_{ij}$ are defined similarly.

We will show, in Section~\ref{sec:computingF}, that the components $v_{ijk}$ can be computed recursively so that, for each
$(i,j)$, the computation of $\mathcal F_{ij}$ requires only
$O(nr)$ flops.

Moreover, we will show, in Section~\ref{sec:solvingB}, that the system
$M_{ij} \mathcal X_{ij} = \mathcal{E}_{ij} - \mathcal{F}_{ij}$, once
the right-hand side term has been computed, can be solved in linear
time, that is in $O(r)$ flops, thanks to the special structure
of the matrix $M_{ij}$.

With all the above tools we can formulate Algorithm~\ref{alg:fij} to
compute the solution of a periodic system of $r$ generalized Sylvester
equations whose coefficients are in upper and lower triangular form as
in Section~\ref{sec:triangular}. Besides the computation of the
solution $X_k$, the routine also computes the matrices
$X_k B_k$ and $X_{k+1} D_k$, here denoted $X^B_k$ and $X^D_k$, respectively,
which are needed for an efficient computation of the right-hand side
$\mathcal E_{ij}-\mathcal F_{ij}$ of the linear system. 

\begin{algorithm}
	\caption{Solution of a periodic system of generalized $\star$-Sylvester equations
	}
	\begin{algorithmic}[1]
		\Procedure{GeneralizedStarSylvesterSystem}{$A_k, B_k, C_k, D_k, E_k$}
		\For{$k = 1, \ldots ,r$}
		\State $X_k \gets 0_{n \times n}$ \Comment{we store the solution here}
		\State $X^B_k \gets 0_{n \times n}$ \Comment{storage for $X_k^B$}
		\State $X^D_k \gets 0_{n \times n}$ \Comment{storage for $X_k^D$}
		\EndFor
		\For{$(i,j) \in \{1,2,\dots,n\}^2$ with $i\geq j$, in the reverse of the ordering~\eqref{eq:list}}
			\State $\mathcal F_{ij} \gets \Call{ComputeF}{X_k,X^B_k, X^D_k,A_k,B_k,C_k,D_k, i, j}$
			\State $x \gets \Call{SolveIntermediateSystem}{M_{ij}, \mathcal E_{ij} - \mathcal F_{ij}}$ 
			\For{$k = 1, \ldots, r$}
				\State $[ X_k ]_{ij} \gets x_k$
				\State $[ X^B_k ]_{ij} \gets ( e_i^\top X_k ) ( B_k e_j )$ 
				\State $[ X^D_k ]_{ij} \gets ( e_i^\top X_{k+1} ) ( D_k e_j )$ \Comment{with the convention $X_{r+1}=X_1^\star$}
				\If{$i \neq j$}
				\State $[ X_k ]_{ji} \gets x_{r+k}$
				\State $[ X^B_k ]_{ji} \gets ( e_j^\top X_k ) ( B_k e_i )$ 
				\State $[ X^D_k ]_{ji} \gets ( e_j^\top X_{k+1} ) ( D_k e_i )$ \Comment{with the convention $X_{r+1}=X_1^\star$}
				\EndIf
			\EndFor
		\EndFor
		\State \Return $X_k$
		\EndProcedure
	\end{algorithmic}
	\label{alg:fij}
\end{algorithm}

Section~\ref{sec:computingF} is devoted to describe the routine
\Call{ComputeF}{}, that computes $\mathcal F_{ij}$ in the right-hand side of the systems $\mathbb S_{ij}$, while Section~\ref{sec:solvingB} describes
the solution of the system, that is the routine
\Call{SolveIntermediateSystem}{}.
An algorithmic description
of the former is given in Algorithm~\ref{alg:computeF}, while
the latter procedure is outlined in algorithmic form in the proof of
Lemma~\ref{lem:solvingB}. A FORTRAN implementation of the algorithm
is available at \url{https://github.com/numpi/starsylv/}. 

%%%%%%%%%%%%%%%%%%%%%%%%%%%%%%%%
%%%%%%%%%%%%%%%%%%%%%%%%%%%%%%%%
%%%%%%%%%%%%%%%%%%%%%%%%%%%%%%%%
%%%%%%%%%%%%%%%%%%%%%%%%%%%%%%%%
%%%%%%%%%%%%%%%%%%%%%%%%%%%%%%%%

\subsection{Computing the term $\mathcal{F}_{ij}$}
\label{sec:computingF}

The computation of the term $\mathcal{F}_{ij}$, if evaluated directly
using Equation~(\ref{eq:Fij}), requires $O(n^2r)$
multiplications and additions. However, by reusing some intermediate quantities
computed in the previous steps, the computation can be carried out in
$O(nr)$ flops.

We rearrange
the first term in the definition of $v_{ijk}$ (and similarly for $v_{jik}$) as follows:
\begin{align*}
  \sum_{\substack{s \geq i, t \geq j \\ (s,t) \neq (i,j)}} ( A_k )_{is} ( X_k )_{st} ( B_k )_{tj} &= 
    \sum_{t > j} ( A_k )_{ii} ( X_k )_{it} ( B_k )_{tj} + \sum_{\substack{s > i, t \geq j}} ( A_k )_{is} ( X_k )_{st} ( B_k )_{tj}. 
\end{align*}
The first summand in the right-hand side of the above equation can be computed in $O(n)$ flops for a given
$k$, so we only need to deal with the efficient evaluation of the latter summand. We can re-arrange
it  as follows:
\[
	\sum_{s>i, t \geq j} ( A_k )_{is} ( X_k )_{st} ( B_k )_{tj} =
	\sum_{s > i} ( A_k )_{is} \underbrace{\sum_{t \geq j} ( X_k )_{st} ( B_k )_{tj}}_{=: (X^B_{k})_{sj}} =:
	\sum_{s > i} ( A_k )_{is} (X^B_{k})_{sj},
\]
and this can be computed in $O(n)$ flops if $(X_k^B)_{sj}$, for $s>i$,
is known. After solving the block with indices $\mathcal{L}_{ij}$, we 
compute and store $(X_k^B)_{ij}$ and $(X_k^B)_{ji}$ (if different), and use them in the
subsequent steps. Notice that the computation of $(X_k^B)_{ij}$
requires only $O(n)$ operations since $(X^B_k)_{ij}$ is the element in
position $(i,j)$ of the product $X_k B_k$, and it depends only on entries of
$X_k$ that have already been computed, thanks to the triangular
structure of $B_k$.
 
In Algorithm~\ref{alg:fij},
$(X_{k}^B)_{sj}$ has been precomputed in the previous steps, 
after the computation of $(X_k)_{sj}$.
Thus, we
can evaluate the first addend of $v_{ijk}$ by computing a
summation of $O(n)$ elements, so by means of $O(n)$ flops.

Setting $X_{r+1}:=X_1^\star$, a similar formula holds for the second term, which can be written
as
\begin{align*}
  \sum_{\substack{s \geq i, t \geq j \\ (s,t) \neq (i,j)}} ( C_k )_{is} ( X_{k+1} )_{st} ( D_k )_{tj} &= 
    \sum_{t > j} ( C_k )_{ii} ( X_{k+1} )_{it} ( D_k )_{tj} + \sum_{\substack{s>i}} ( C_k )_{is} \underbrace{\sum_{t>j}( X_{k+1} )_{st} ( D_k )_{tj}}_{:=(X_k^D)_{sj}}, 
\end{align*}
and can be computed in ${O}(n)$ by storing the computed
$(X_k^D)_{sj}$ at every step, as with $(X^B_k)_{sj}$.

An algorithmic description of the above process, which can be plugged
directly into Algorithm~\ref{alg:fij}, is given in Algorithm~
\ref{alg:computeF},
and clearly requires $O(nr)$ arithmetic operations. Notice that in Algorithm~\ref{alg:computeF} all scalar products
are computed on the complete rows and columns of the matrices $X_1,\ldots,X_r$. This is
for notational convenience, but the formulation of Algorithm \ref{alg:computeF} is equivalent to \eqref{eq:Fij},
thanks to the initialization to zero of $X_k, X_k^B$, and $X^D_k$, for $k=1,\ldots,r$.
Nevertheless, in the implementation it is convenient to skip all
the entries that are known to be zero. 

\begin{algorithm}
  \caption{Subroutines used to compute the entries of $\mathcal F_{ij}$, which is
    part of the right-hand side of the linear system.}
  \begin{algorithmic}[1]
	\Procedure{ComputeF}{$X_k,X_k^B, X_k^D, A_k, B_k, C_k, D_k,i, j$}
	\If{$i = j$}
		\State $F \gets$ \Call{ComputeW}{$X_k, X_k^B, X_k^D, A_k, B_k, C_k, D_k,i, j$}
	\Else
		\State $F(1:r) \gets$ \Call{ComputeW}{$X_k, X_k^B, X_k^D, A_k, B_k, C_k, D_k,i, j$}
		\State $F(r+1:2r) \gets$ \Call{ComputeW}{$X_k, X_k^B, X_k^D, A_k, B_k, C_k, D_k,j, i$}
	\EndIf
    \State \Return $F$
    \EndProcedure
    \Procedure{ComputeW}{$X_k,X_k^B, X_k^D, A_k, B_k, C_k, D_k,i, j$}
    \State $F \gets 0_{r}$
    \For{$k = 1, \ldots, r$}
    \State $f_1 \gets (A_k)_{ii} (e_i^\top X_k) (B_k e_j) +  (e_i^\top A_k) (X^B_{k} e_j)$
    \State $f_2 \gets (C_k)_{ii} (e_i^\top X_{k+1}) (D_k e_j) + (e_i^\top C_k) (X^D_k        e_j)$ \Comment{With $X_{r+1}=X_1^\star$}
    \State $F_{k} \gets f_1 + f_2$
    \EndFor
    \State \Return $F$
    \EndProcedure
  \end{algorithmic}
  \label{alg:computeF}
\end{algorithm}

\begin{remark}
  In Algorithm~\ref{alg:fij}, we have shown that it is possible
  to compute $(X_{k}^B)_{ij}$ and $(X^D_k)_{ij}$ after the
  solution of the linear system. In fact, a careful look at the
  algorithm shows that the scalar products
  \[
    [ X^B_k ]_{ij} \gets ( e_i^T X_k ) ( B_k e_j ), \qquad 
    [ X^D_k ]_{ij} \gets ( e_i^T 
    X_{k+1}              ) ( D_k e_j )
  \]
  can be avoided. All non-zero elements in the above summations,
  except the ones corresponding to the diagonal entries of $X_k$
  and $B_k$ or $D_k$, are already computed and summed up in
  \Call{ComputeF}{}. Thus, the entries in position
  $(i,j)$ of $X^B_k$ and $X^D_k$ can be computed with
  an $O(1)$ update of these partial sums. This does not
  change the asymptotic cost, but slightly improves the
  timing and it has been exploited in the implementation. However,
  we decided to avoid describing it in detail in the pseudocode for the
  sake of simplicity. 
\end{remark}

\begin{remark}
For simplicity, both here in the pseudocode and in the implementation used in the experiments, we have allocated $2rn^2$ additional memory entries to store the matrices $X_k^B$ and $X_k^D$. However, it is possible to implement the algorithm allocating with only $O(r+n)$ additional memory if one can overwrite the input matrices $A_k,B_k,C_k,D_k,E_k$. Indeed, while computing the periodic Schur form as described in Lemma~\ref{thm:lemtriang}, one can use the upper triangular part of $A_k,B_k,C_k,D_k$ to store $\widehat{A}_k,\widehat{B}_k,\widehat{C}_k,\widehat{D}_k$ and their lower triangular parts to store in compressed format the orthogonal matrices $Q_k,\widehat{Q}_k, Z_k, \widehat{Z}_k$. Then, one overwrites $E_k$ with $\widehat{E}_k$. Afterwards, the 
matrices $Q_k,\widehat{Q}_k$ are not needed anymore, and with some index juggling one can overwrite the $rn(n-1)$ entries used to store them with the entries of $X_k^B$ and $X_k^D$, discarding those that are not needed anymore. 
The entries of the solution $X_k$ can overwrite those of $\widehat{E}_k$.
\end{remark}

%%%%%%%%%%%%%%%%%%%%%%%%%%%%%%%%
%%%%%%%%%%%%%%%%%%%%%%%%%%%%%%%%
%%%%%%%%%%%%%%%%%%%%%%%%%%%%%%%%
%%%%%%%%%%%%%%%%%%%%%%%%%%%%%%%%
%%%%%%%%%%%%%%%%%%%%%%%%%%%%%%%%

\subsection{Solving the small linear systems}
\label{sec:solvingB}

We describe how to efficiently solve the linear system \eqref{eq:mxij} involving
the matrix $M_{ij}$. The cases $i = j$ and $i \neq j$
are different in the dimension of the matrix, but share
the same structure, so we can handle them at the same time.
More precisely, we have the following result for $\star = \top$. 
\begin{lemma}
  \label{lem:solvingB}
  Let $M$ be an $\ell \times \ell$
  matrix such that the elements in position $(i,j)$
  are allowed to be nonzero only if $0 \leq j - i \leq 1$ or if $(i,j) = (\ell, 1)$.
  Then $M$ admits a QR factorization $M = QR$
  where $R$ is upper bidiagonal except in the last column, and $Q$ is
  a product of $\ell - 1$ plane rotations. 
\end{lemma}

\begin{proof}
	The proof is constructive and by induction. The case $\ell = 1$ is
	trivial, so let us assume that we have an $(\ell + 1) \times (\ell + 1)$
	matrix $M$, so that we can compute a rotation $G$ acting on
	the first and last row that annihilates the elements in position
	$(\ell + 1, 1)$. More precisely
	\[
		GM=G \begin{bmatrix}
			\times & \times \\
			& \ddots & \ddots \\
			& & \times & \times \\
			\times &&& \times \\
		\end{bmatrix} =
		\left[ \begin{array}{c|ccc}
			a_1 & b_1 && x_1 \\ \hline
			&  &  \\
			& & \widetilde M & \\
			& &&  \\      
		\end{array} \right],
\]
where $\widetilde M$ has the same shape as $M$,
but is of size $\ell \times \ell$. Therefore, we can factorize
$\widetilde M= \widetilde Q \widetilde R$, with $\widetilde Q$ being
the product of $\ell - 1$ rotations. Setting
$Q := G^\star \left[\begin{smallmatrix} 1 & 0 \\ 0 & Q\end{smallmatrix}\right]$
and
\[
	R =     \left[ \begin{array}{c|ccc}
			a_1 & b_1 && x_1 \\ \hline
			&  &  \\
			& & \widetilde R & \\
			& &&  \\      
		\end{array} \right]
\]
concludes the proof. 
\end{proof}

The above proof shows that the matrices $Q$ and $R$ can be computed in
$O(\ell)$, and then the linear system $Mx = QRx=  y$ can be solved in
$O(\ell)$ by the application of $O(\ell)$ rotations to $y$ (each of these operations can be done in $O(1)$) and by a back substitution, that, thanks to the sparsity of $R$, can be computed in $O(\ell)$ as well.

In our case the matrix of the linear system
has $\ell \in \{ r, 2r \}$, so
we can solve each intermediate linear system in $O(r)$.

The case $\star = \ctop $ is not much different, since the
matrices $M_{ij}$ of the linear system are block bidiagonal
(except for the block at the end of the first column), with $2 \times 2$
blocks. In fact, the matrices $M_{ij}$ can be brought into upper
triangular form using about $5 r$ rotations, and the upper
triangular form enjoys a block bidiagonal form that allows us to
solve the linear system in $O(r)$. 

 Lemma \ref{lem:solvingB} can be easily converted into a routine
and provides a possible implementation for \Call{SolveIntermediateSystem}{}
in Algorithm~\ref{alg:fij}. An implementation for this routine
can be found in the code used for the tests, available at
\url{https://github.com/numpi/starsylv/}.

%%%%%%%%%%%%%%%%%%%%%%%%%%%%%%%%
%%%%%%%%%%%%%%%%%%%%%%%%%%%%%%%%
%%%%%%%%%%%%%%%%%%%%%%%%%%%%%%%%
%%%%%%%%%%%%%%%%%%%%%%%%%%%%%%%%
%%%%%%%%%%%%%%%%%%%%%%%%%%%%%%%%

\subsection{Computational cost and storage}
\label{sec:computational-cost}

We evaluate the total computational cost of the algorithm (in terms of floating-point operations)
by taking into account the cost of all single steps.

Step~\ref{it:alg-reduction-periodic} requires only some bookkeeping and possibly swapping and transposing matrices in memory, but no floating point operations. This step produces several periodic systems; let $r_1,r_2,\dots, r_m$ be their sizes, with $r_1+\dots+r_m \leq r$. We prove that each of these systems is solved using $O(n^3r_i)$ flops.

Step~\ref{it:alg-reduction-triangular} (for the $i$th periodic system of size $r_i$) requires computing a periodic Schur form, which costs $O(n^3r_i)$ with the algorithm of~\cite{bojanczyk1992periodic}. Once the periodic Schur form has been computed, the changes of variables amount to $O(r_i)$ products between $n\times n$ matrices.

In Step~\ref{it:alg-backsubstitution}, the method described in Section~\ref{sec:computingF} allows one to compute each of the $\frac{n(n+1)}{2}$ terms $\mathcal{F}_{ij}$ in $O(nr_i)$ flops, and Section~\ref{sec:solvingB} shows how to solve in $O(r_i)$ flops the linear systems required in each of the $\frac{n(n+1)}{2}$ back substitution steps. The total amount of flops required by this step is, thus, $O(n^3r_i)$.

Step~\ref{it:alg-remaining} requires applying formula~\eqref{eq:alfak} (which costs $O(n^3)$ to compute) once for each remaining variable, that is, at most $r-1$ times.

Combining all the above steps we obtain an algorithm with a total cost of
$O(n^3 r)$ flops. Moreover, the only storage required
is the one of $O(r)$ matrices of size $n \times n$, so the storage
required is $O(n^2 r)$, which is optimal (given that the same amount
of storage is required to store the solutions).

\begin{remark} \label{rem:graphtheory}
	Step~\ref{it:alg-reduction-periodic} requires some discrete computations on the indices to identify the periodic systems and eliminate variables and equations; we have ignored them here since they involve no floating-point operations, but we have shown in Section \ref{sec:algored} that they can be performed in $O(r)$ operations with the help of a graph algorithm.
\end{remark}

%%%%%%%%%%%%%%%%%%%%%%%%%%%%%%%%
%%%%%%%%%%%%%%%%%%%%%%%%%%%%%%%%
%%%%%%%%%%%%%%%%%%%%%%%%%%%%%%%%
%%%%%%%%%%%%%%%%%%%%%%%%%%%%%%%%
%%%%%%%%%%%%%%%%%%%%%%%%%%%%%%%%

\subsection{Backward error analysis}
\label{sec:backward-error}

Here we provide a backward error analysis of
the algorithm described in the previous sections. We use the standard floating point number model with unit roundoff $u$ and, for an expression $\ell$, we denote by $\fl(\ell)$ the computed value of $\ell$ using floating point operations. We use the notation 
$$
\gamma_k:=\frac{cku}{1-cku},
$$
where $c$ denotes a small constant, whose exact value is not relevant (see \cite[p. 68]{higham-accuracy}).

We assume that all linear systems $Ax=b$ that are encountered
are solved using a backward stable method. More precisely, we say that an algorithm to solve a linear system $Ax=b$, with $A\in\C^{m\times m}$, has {\em backward error $\varepsilon_A$} if the computed solution $\tilde{x}=\fl(A^{-1}b)$ is the exact solution of a perturbed system $(A+\delta A)\tilde{x}=b$, with $\norm{\delta A}_2/\norm{A}_2\leq\varepsilon_A$. Note that only the coefficient matrix is perturbed (see \cite[Th. 19.5]{higham-accuracy} and the following discussion for an explanation). In the case of solving the system with the QR factorization using $s$ Givens rotations, as we do in Section~\ref{sec:solvingB} with $s=O(r)$,
this quantity can be taken as
$
\varepsilon_A=m\cdot\gamma_{s}
$
(see p. 368 and Theorem 19.10 in \cite{higham-accuracy}). 
The factor $m$ comes from the fact that the bound in \cite{higham-accuracy} is only given column-wise and 
\begin{equation}\label{normbound}
\norm{{\rm Col}_j A}_2\leq\norm{A}_2\leq\sqrt{m}\norm{A}_1=\sqrt{m}\max_{j=1,\hdots,m}\norm{{\rm Col}_jA}_1\leq
m\max_{j=1,\hdots,m}\norm{{\rm Col}_jA}_2,
\end{equation}
for all $j=1,\hdots,m$, where $\mathrm{Col}_jA$ is the $j$th column of $A$ (see, for instance, \cite[Tables 6.1 and 6.2]{higham-accuracy} for the last two inequalities).

We obtain a backward error result formulating the problem as a vectorized linear system. For simplicity, we will focus on periodic systems with upper and lower triangular coefficients in Theorem \ref{th:backstab}. The general case will be commented right after the proof. 

\begin{theorem} \label{th:backstab}
	Consider a system of equations of the form~\eqref{eq:periodic}, with $A_k,C_k,B_k^\top,D_k^\top$ being upper triangular, and
	let $M\mathcal{X}=\mathcal{E}$ be its vectorized form, where $M\in\mathbb{C}^{rn^2\times rn^2}$ 
	if $\star=\top$, or $M\in\mathbb{R}^{2rn^2\times 2rn^2}$ if $\star=\ctop $.

	When implemented in standard floating-point arithmetic, the algorithm described 
	in Sections~{\rm\ref{sec:alg-backsubstitution}--\ref{sec:solvingB}} produces a result $\tilde{\mathcal{X}}$ satisfying
	\begin{equation} \label{eq:backstab}
		(M+\delta M)\tilde{\mathcal{X}} = \mathcal{E}+\delta \mathcal{E},
	\end{equation}
	with $\norm{\delta M}_2 / \norm{M}_2\leq {r}\,\gamma_r+\gamma_{n^2}(1+{r}\,\gamma_r)$, $\norm{\delta \mathcal{E}}_2 / \norm{\mathcal{E}}_2\leq\gamma_{n^2}$.
\end{theorem}
\begin{remark}\label{rem:residual}
	The reader may wonder if a stronger form of structured backward stability holds:
	the algorithm should produce matrices that satisfy
	\[
		(A_k + \delta A_k)\tilde{X}_{\alpha_k}^{s_k}(B_k+\delta B_k) - (C_k+\delta C_k) \tilde{X}_{\beta_k}^{t_k} (D_k+\delta D_k) = E_k + \delta E_k \quad k=1,\dots, r,
	\]
	with $\norm{\delta S_k}_2/\norm{S_k}_2$ being small, for $S=A,B,C,D,E$. Unfortunately, algorithms of this family fail to be structurally backward stable even in the simplest case of a single
	Sylvester equation $AX-XD=E$, as shown in~\cite[\S 16.2]{hig93} (see also the discussion in~\cite{byers-rhee} for the case $s=1$).

        Note that Theorem~\ref{th:backstab} is nevertheless sufficient to show that the residual of each equation $R_k = \norm{A_k \tilde{X}_{\alpha_k}^{s_k}B_k - C_k\tilde{X}_{\beta_k}^{t_k}D_k - E_k}_F$, for $k=1,2,\dots,r$, is small. Indeed, $\norm{M\tilde{\mathcal{X}}-\mathcal{E}}_2 = \sqrt{\sum_{k=1}^r R_k^2}$ satisfies (see~\cite[Thm~7.1]{higham-accuracy})
\[
\frac{\norm{M\tilde{\mathcal{X}}-\mathcal{E}}_2}{\norm{M}_2\norm{\tilde{\mathcal{X}}}_2 + \norm{\mathcal{E}}_2} \leq \max\left(
\frac{\norm{\delta M}_2}{\norm{M}_2}, \frac{\norm{\delta \mathcal{E}}_2}{\norm{\mathcal{E}}_2}.
\right)
\]
\end{remark}

In order to prove Theorem~\ref{th:backstab}, we need the following technical results.

\begin{lemma}\label{lem:deltad}
Let $N\in\C^{m\times m}$ and $x,y\in\C^m$, with $x,y\neq0$, be such that
\begin{equation}\label{deltan}
y=(N+\Delta N)x,\qquad \frac{\norm{\Delta N}_2}{\norm{N}_2}\leq\varepsilon,
\end{equation}
 for some $\varepsilon>0$. Let $\delta y\in\C^m$ be such that 
\begin{equation}\label{deltay}
\frac{\norm{\delta y}_2}{\norm{y}_2}\leq \kappa,
\end{equation}
for some $\kappa>0$. Then
$
y+\delta y=(N+\delta N)x,
$
for some $\delta N\in\C^{m\times m}$ with 
$
\frac{\norm{\delta N}_2}{\norm{N}_2}\leq \varepsilon+\kappa(1+\varepsilon).$
\end{lemma}
\begin{proof}
From \eqref{deltan} and \eqref{deltay} we get
\begin{equation}\label{deltaybound}
\norm{\delta y}_2\leq\kappa\norm{y}_2\leq\kappa\left(\norm{N}_2+\norm{\Delta N}_2\right)\norm{x}_2\leq\kappa(1+\varepsilon)\norm{N}_2\norm{x}_2.
\end{equation}

Now, setting $\widetilde N:=\norm{x}_2^{-2}\cdot(\delta y) x^\ctop $, we have $\widetilde N x=\delta y$ and $\norm{\widetilde N}_2=\norm{\delta y}_2/\norm{x}_2$, so $\norm{\delta y}_2=\norm{\widetilde N}_2\norm{x}_2$. Then, by \eqref{deltaybound},
\begin{equation}\label{widetilden}
\norm{\widetilde N}_2\leq\kappa(1+\varepsilon)\norm{N}_2.
\end{equation}
Finally, taking $\delta N:=\Delta N+\widetilde N$, and using \eqref{widetilden}, we arrive at
$
\norm{\delta N}_2\leq\norm{\Delta N}_2+\norm{\widetilde N}_2\leq(\varepsilon+\kappa(1+\varepsilon))\norm{N}_2.
$
\end{proof}

\begin{lemma} \label{lem:backstab}
	Consider a square linear system of the form 
	$
		Fx = b - \sum_{k=1}^s N_k c_k,
	$
	where $F,N_k\in\C^{m\times m}$, and $b,c_k\in\C^m$ are given, for $k=1,\dots,s$, and $x$ is the unknown.

	Forming the sum in the right-hand side, in floating point arithmetic, and then solving the linear system using an algorithm with backward error $\varepsilon_F$, produces a computed solution $\tilde{x}$ which is the exact solution
	of a perturbed system
	\[
		(F+\delta F)\tilde{x} = b+\delta b - \sum_{k=1}^s (N_k+\delta N_k) c_k,
	\]
	with
	 $$
	 \frac{\norm{\delta F}_2}{\norm{F}_2}\leq\varepsilon_F,
	 \quad \frac{\norm{\delta b}_2}{\norm{b}_2}\leq\gamma_s,\quad
	  \frac{\norm{\delta N_k}_2}{\norm{N_k}_2}\leq m\gamma_m+\gamma_s(1+m\gamma_m).
	 $$
\end{lemma}
\begin{proof}
	Let $\tilde{d}_k = \fl(N_k c_k)$, $\tilde{f}=\fl(b - \sum_{k=1}^s \tilde{d}_k)$. By hypothesis,
	$(F+\delta F)\tilde{x} = \tilde{f}$, with  $\norm{\delta F}_2/\norm{F}_2\leq\varepsilon_F$. 
	The usual backward error analysis of summation can be used to show that 
	$\tilde{f}  = b + \delta b - \sum_{k=1}^s (\tilde{d}_k  + \delta \tilde{d}_k)$, 
	with $|(\delta b)_i|/|b_i|,|(\delta\widetilde d_k)_i|/|(\widetilde d_k)_i|\leq\gamma_s$, for $i=1,\hdots,m$ (see \cite[Section~4]{higham-accuracy}). Now, by standard backward error analysis of matrix-vector multiplication, we know that
	$
	\tilde{d}_k=(N_k+\Delta N_k)c_k,
	$
	with $\norm{{\rm Col}_j (\Delta N_k)}_2/\norm{{\rm Col}_j(N_k)}_2\leq\gamma_m$, for $j=1,\hdots,m$ (see \cite[Section~3.5]{higham-accuracy}). Using \eqref{normbound}, this implies $\norm{\Delta N_k}_2/\norm{N_k}_2\leq m\gamma_m$. Now, we can apply Lemma \ref{lem:deltad}, with $y=\tilde{d}_k,\delta y=\delta\tilde{d}_k$,  $x=c_k$, $N=N_k$ and $\Delta N=\Delta N_k$, to conclude that
		$
		\tilde{d}_k  + \delta \tilde{d}_k = (N_k+\delta N_k)c_k,
		$
		 with $\norm{\delta N_k}_2/\norm{N_k}_2\leq m\gamma_m+\gamma_s(1+m\gamma_m)$, as wanted.
\end{proof}

\begin{proof}[Proof of Theorem~{\rm\ref{th:backstab}}]
	
	We note that each step of the block back substitution corresponds to
	solving a linear system of the form~\eqref{eq:mxij}. More precisely, this system is
$$
M_{ij}\mathcal{X}_{ij}=\mathcal{E}_{ij}-\sum_{(s,t)\in\mathcal U_{ij}}N_{st}^{(ij)}\mathcal{X}_{st},	
$$
	where $\mathcal U_{ij}=\{(i',j')\,:\, \max\{i',j'\}\ge \max\{i,j\}\mbox{ and }\min\{i',j'\}\ge \min\{i,j\}\}$ and the matrices $N_{st}^{(ij)}$ are given by writing \eqref{eq:Fij} in matrix form. By Lemma~\ref{lem:backstab}, there are some matrices $\delta M_{ij}$ and $\delta N_{st}^{(ij)}$ such that
	$
	(M_{ij}+\delta M_{ij})\mathcal{\widetilde X}_{ij}=\mathcal{E}_{ij}+\delta\mathcal{E}_{ij}-\sum_{(s,t)\in\mathcal U_{ij}}(N_{st}^{(ij)}+\delta N_{st}^{(ij)})\mathcal{\widetilde X}_{st},
	$
	where $\mathcal{\widetilde X}_{ij}$ are the computed solutions at the $(i,j)$ step and $\mathcal{\widetilde X}_{st}$, for $s\ge i, t\ge j$, with $(s,t)\neq(i,j)$, are the ones computed in the previous steps, and
	$$
	\frac{\norm{\delta M_{ij}}_2}{\norm{M_{ij}}_2}\leq\varepsilon_{M_{ij}},\quad\frac{\norm{\delta N_{st}^{(ij)}}_2}{\norm{N_{st}^{(ij)}}_2}\leq r\gamma_r+\gamma_{n^2}(1+r\gamma_r),
	\quad\frac{\norm{\delta\mathcal{E}_{ij}}_2}{\norm{\mathcal{E}_{ij}}_2}\leq\gamma_{n^2}.
	$$
	If the $r\times r$ (or $(2r)\times(2r)$) linear system is solved through the QR factorization of $M_{ij}$, then 	$\varepsilon_{M_{ij}} \leq r\gamma_{r}$, as mentioned before (see \cite[Th. 19.10]{higham-accuracy}).
	
	This gives a backward error for each block-row of the matrix $M$ and of the 
	right-hand side $\mathcal{E}$ in Theorem~\ref{th:backstab}. Since these rows are
	never reused between equations, this defines a perturbation of 
	$M$ and $\mathcal{E}$ which ensures~\eqref{eq:backstab}.
\end{proof}

We note that Theorem \ref{th:backstab} corresponds to Step~\ref{it:alg-backsubstitution} in the procedure described at the beginning of Section~\ref{sec:algorithm} for solving a general system \eqref{eq:gensystem}. The remaining steps can be carried out also in a backward stable way, as we are going to explain. 

	Step~\ref{it:alg-reduction-periodic} involves no computations, just relabeling of the equations,
	transpositions and conjugations (which are exact in floating point arithmetic).

	Step~\ref{it:alg-reduction-triangular} is backward stable since the periodic QZ algorithm relies on
	unitary transformations and the following change of variables is unitary. 

In Step~\ref{it:alg-remaining}, the vectorization of~\eqref{eq:alfak}
	produces the linear system
	\[
		(B_k^{\top} \otimes A_k) \operatorname{vec}(X_{\alpha_k}^{s_k}) = \operatorname{vec}(E_k) + (D_k^\top \otimes C_k)\operatorname{vec}(X_{\beta_k}^{t_k}),
	\]
	which is again in the form treated in Lemma~\ref{lem:backstab}, so we only have to ensure that the
	method used to solve this linear system of the form $(B_k^{\top} \otimes A_k)\operatorname{vec}(X)=\operatorname{vec}(F)$
	is backward stable. To solve this system, we first compute $\tilde{Y} = \fl(A_k^{-1}F)$
	column by column, each time solving a linear system with $A_k$, and then similarly 
	$\tilde{X}=\fl(\tilde{Y}B_k^{-1})$, solving a linear system for each of its rows. 

	We assume that the linear systems with $A_k$ are solved with a backward stable method, i.e.,
	\[
	(A_k + \delta_j A_{k})\mathrm{Col}_j(\tilde{Y})=\mathrm{Col}_j(F), \quad \frac{\norm{\delta_j A_k}_2}{\norm{A_k}_2} \leq \varepsilon_{A_k},
	\]
	(note that there is a different perturbation $\delta_j A_k$ for each $j$); hence we have
	\[
		(\mathbb{A}+\delta\mathbb{A})\operatorname{vec}(\tilde{Y}) = \operatorname{vec}(F), \quad \frac{\norm{\delta\mathbb{A}}_2}{\norm{\mathbb{A}}_2} \leq \varepsilon_{A_k},
	\]
	where $\mathbb{A}=I_n \otimes A_k$ and $\delta \mathbb{A} = \operatorname{diag}(\delta_1 A_k,\dots,\delta_n A_k)$.

	An analogous argument shows that
	\[
		(\mathbb{B}+\delta\mathbb{B})\operatorname{vec}(\tilde{X}) = \operatorname{vec}(\tilde{Y}),\quad \frac{\norm{\delta\mathbb{B}}_2}{\norm{\mathbb{B}}_2} \leq \varepsilon_{B_k},
	\]
	where $\mathbb{B}=B_k^\top \otimes I_n$. Combining these two relations we have
	$
		\operatorname{vec}(F) = (\mathbb{A}+\delta\mathbb{A})(\mathbb{B}+\delta\mathbb{B})\operatorname{vec}(\tilde{X})
		= (\mathbb{A}\mathbb{B}+ \delta(\mathbb{AB}))\operatorname{vec}(\tilde{X}),
	$
	with $\delta(\mathbb{AB}) = \delta\mathbb{A}\cdot\mathbb{B}+\mathbb{A}\cdot\delta\mathbb{B} + \delta \mathbb{A}\cdot\delta\mathbb{B}$. 	
	Since $\norm{\mathbb{A}\mathbb{B}}_2 = \norm{\mathbb{A}}_2\norm{\mathbb{B}}_2$ for our choice of $\mathbb{A}$ and $\mathbb{B}$ (thanks to the properties of the Kronecker product \cite[p. 253]{hj-topics}), we get
	\begin{align*}
		\frac{\norm{\delta (\mathbb{A}\mathbb{B})}_2}{\norm{\mathbb{A}\mathbb{B}}_2} &=
		\frac{\norm{\delta\mathbb{A}\cdot\mathbb{B}+\mathbb{A}\cdot\delta\mathbb{B} + \delta \mathbb{A}\cdot\delta\mathbb{B}}_2}{\norm{\mathbb{A}}_2\norm{\mathbb{B}}_2} \leq 
		\frac{\norm{\delta\mathbb{A}}_2\norm{\mathbb{B}}_2+\norm{\mathbb{A}}_2\norm{\delta\mathbb{B}}_2 + \norm{\delta \mathbb{A}}_2\norm{\delta\mathbb{B}}_2}{\norm{\mathbb{A}}_2\norm{\mathbb{B}}_2}\\
		& \leq \varepsilon_{A_k} + \varepsilon_{B_k} + \varepsilon_{A_k}\varepsilon_{B_k}.
	\end{align*}

As a consequence of these arguments, the procedure described at the beginning of Section~\ref{sec:algorithm} produces a backward stable algorithm for solving general systems of the form \eqref{eq:gensystem}.

\subsection{Numerical experiments}
\label{sec:numerical-experiments}

We have implemented the proposed algorithm for the solution
in the case $\star = \top$. The case $\star=\ctop $ can be obtained
with minimal changes (from the algorithmic point of view), so
we decided to avoid running the same experiments concerning
stability and performance. We have run the tests on a server
with a Xeon X5680 CPU and 24 GB of memory. Our implementation
is available at \url{https://github.com/numpi/starsylv/}.
The code has been compiled with GNU Fortran compiler and
linked with the (single-threaded) BLAS reference
implementation (\texttt{libblas.so}, \url{http://www.netlib.org/blas/}).

We have computed the CPU time required by our implementation as a function
of the size of the matrices $n$ and of the number of equations
in the reduced system $r$, and we have compared it with the behavior predicted by our analysis. We have considered only systems with
triangular factors. The general case requires the reduction to triangular factors through the periodic Schur form as described in Section~\ref{sec:red2}, which 
has been already implemented in \cite[subroutines~\texttt{MB03BD} and~\texttt{MB03BZ}]{benner1997slicot} (see also~\cite{bojanczyk1992periodic,kressner01periodic}).

\begin{figure}
	\centering
	\begin{tikzpicture}
		\begin{loglogaxis}[
			xlabel = $n$, ylabel = Time (s),
			legend pos = north west,
			width = .45\linewidth,
			height = .25\textheight
			]
			\addplot table {times.dat};
			\addplot[dashed,domain=32:8192] { x^3 * 1e-6 };
			\legend{Timing, $O(n^3)$};
		\end{loglogaxis}
	\end{tikzpicture}~~~~~~~~\begin{tikzpicture}
	\begin{loglogaxis}[
	xlabel = $r$, ylabel = Time (s),
	legend pos = north west,
	width = .45\linewidth,
	height = .25\textheight
	]
	\addplot table {timesR.dat};
	\addplot[dashed,domain=32:16384] { x * 5e-4 };
	\legend{Timing, $O(r)$};
	\end{loglogaxis}
	\end{tikzpicture}
	\caption{\footnotesize On the left, 
		the CPU time required by the algorithm described in
		Section~\ref{sec:backward-error} for the $\star = \top$
		case, as a function of $n$.
		The timings reported are for a system with $3$ equations,
		already in the required triangular form. The problems tested
                have sizes ranging from $n = 32$ to $n = 8192$. 
        On the right, the CPU time required by the algorithm described in
        Section~\ref{sec:backward-error} for the $\star = \top$
        case, as a function of $r$.
        The timings reported are for a system with $r$ equations
        and coefficients of size $16 \times 16$,
        already in the required triangular form. The problems tested
        have sizes ranging from $r = 32$ to $r = 16384$. 
                }
	\label{fig:complexity}
\end{figure}
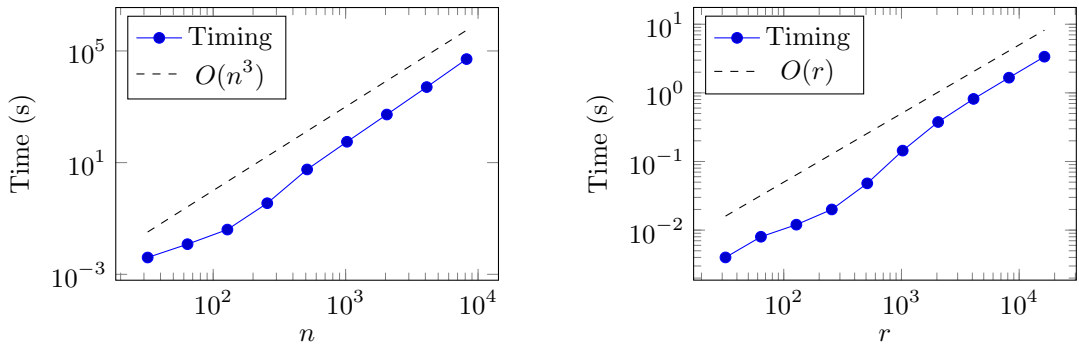

The results are reported in Figure~\ref{fig:complexity}, on the left, for the CPU
time required for the solution of a system of three equations with
coefficients of variable size $n$, and on the right for a system of $r$
equations of size $16$. Both plots confirm the cubic and linear
dependence of the CPU time on the parameters $n$ and $r$,
respectively, that we expect. The dashed lines in the two plots are
obtained plotting the functions $k_nn^3$ and $k_rr$ for two
appropriate constants $k_n$ and $k_r$.

Beside timings, we have also tested the accuracy of the
implementation.  For each value of $n$ and $r$ we have generated
several systems of $\top$-Sylvester equations (in the required
triangular form), and we have computed the residuals
$R_k := \norm{A_k X_k B_k - C_k X_{k+1} D_k - E_k}_F$ for
$k = 1, \ldots, r - 1$, and
$R_{r} := \norm{A_r X_r B_r - C_r X_{1}^\top D_r - E_r}_F$.  Then, the
$2$-norm of the residual of the linear system can be evaluated as
$R := \sqrt{R_1^2 + \cdots + R_r^2}$.
In Figure~\ref{fig:backwardError} we
have plotted an upper bound of the relative residuals $R / \norm{M}_{2}$,
obtained using the relation $n \sqrt{r} \norm{M}_2 \geq \norm{M}_F$,
where $M$ is the matrix of the ``large'' linear system, for different
values of $n$ and $r$ (recall that $M$ has size $n^2 r$).
Each value has been averaged over $100$
runs. The Frobenius norm of $M$ is easily computable recalling that, if two
matrices $M_1$ and $M_2$ do not have non-zero entries in corresponding
positions, then
$\norm{M_1 + M_2}_F^2 = \norm{M_1}_F^2 + \norm{M_2}_F^2$, and the
relation $\norm{A \otimes B}_F = \norm{A}_F \norm{B}_F$.

In these tests, the coefficients matrices $A_k, B_k, C_k, D_k$ have been
chosen with random entries with normal distribution, and with the correct triangular
structure. We have then shifted $A_k$ and $B_k$ with $\sqrt{n} I$ to
avoid finding solutions with very large norms. 

From the tests performed so far, the algorithm behaves in a backward
stable way, as predicted by our analysis. In fact, one can spot
that the error growth with respect to $n$ and $r$
is even less than the upper bound proved in this section. The error seems to grow slightly
less than $\sqrt{n}$, and to be independent of $r$. 
This behavior is often encountered in
dense linear algebra algorithms, since on average the errors do not
accumulate in the same direction (see e.g.~\cite[Section~4.5]{higham-accuracy}).

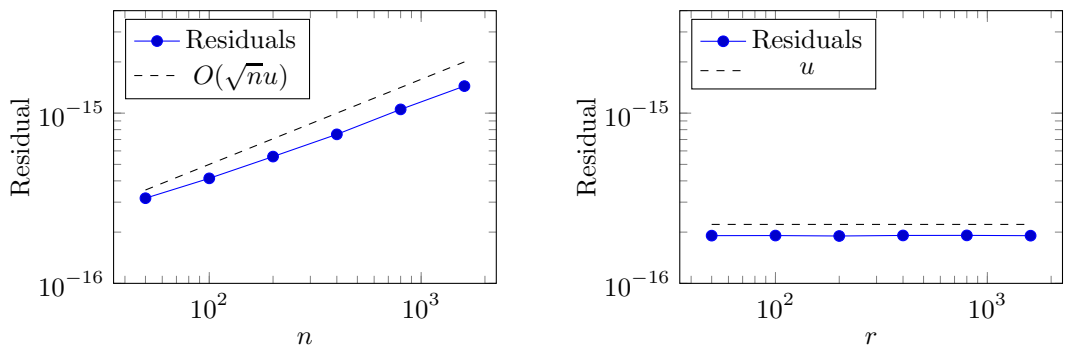
\begin{figure}
	\centering
	\begin{tikzpicture}
		\begin{loglogaxis}[
			xlabel = $n$, ylabel = Residual,
			legend pos = north west,
			width = .45\linewidth,
			height = .25\textheight,
			ymax=4e-15,
			ymin=1e-16
			]
			\addplot table {backwardErrorN.dat};
			\addplot[dashed,domain=50:1600] { sqrt(x) * 5e-17 };
			\legend{Residuals, $O(\sqrt nu)$};
		\end{loglogaxis}
		\end{tikzpicture}~~~~~~~~\begin{tikzpicture}
		\begin{loglogaxis}[
		xlabel = $r$, ylabel = Residual,
		legend pos = north west,
		width = .45\linewidth,
		height = .25\textheight,
		ymin = 1e-16,
		ymax = 4e-15
		]
		\addplot table {backwardErrorR.dat};
		\addplot[dashed,domain=50:1600] {2.22e-16 * 1 };
		\legend{Residuals, $u$};
		\end{loglogaxis}
	\end{tikzpicture}
	\caption{\footnotesize On the left, average residuals of $100$ systems of 
		$\top$-Sylvester equations solved via the algorithm
		described in Section~\ref{sec:algorithm}. The systems considered
		have $3$ equations with a variable coefficient size $n$. On the
	right, average residuals of $100$ systems of
	$\top$-Sylvester equations solved via the algorithm
	described in Section~\ref{sec:algorithm}. The systems considered
	have coefficients with size $ 8\times 8$, and $r$ equations.}
	\label{fig:backwardError}
\end{figure}

%%%%%%%%%%%%%%%%%%%%%%%%%%%%%%%%
%%%%%%%%%%%%%%%%%%%%%%%%%%%%%%%%
%%%%%%%%%%%%%%%%%%%%%%%%%%%%%%%%
%%%%%%%%%%%%%%%%%%%%%%%%%%%%%%%%
%%%%%%%%%%%%%%%%%%%%%%%%%%%%%%%%
%%%%%%%%%%%%%%%%%%%%%%%%%%%%%%%%
%%%%%%%%%%%%%%%%%%%%%%%%%%%%%%%%
%%%%%%%%%%%%%%%%%%%%%%%%%%%%%%%%
%%%%%%%%%%%%%%%%%%%%%%%%%%%%%%%%
%%%%%%%%%%%%%%%%%%%%%%%%%%%%%%%%

%%%%%%%%%%%%%%%%%%%%%%%%%%%%%%%%
%%%%%%%%%%%%%%%%%%%%%%%%%%%%%%%%
%%%%%%%%%%%%%%%%%%%%%%%%%%%%%%%%
%%%%%%%%%%%%%%%%%%%%%%%%%%%%%%%%
%%%%%%%%%%%%%%%%%%%%%%%%%%%%%%%%
%%%%%%%%%%%%%%%%%%%%%%%%%%%%%%%%
%%%%%%%%%%%%%%%%%%%%%%%%%%%%%%%%
%%%%%%%%%%%%%%%%%%%%%%%%%%%%%%%%
%%%%%%%%%%%%%%%%%%%%%%%%%%%%%%%%
%%%%%%%%%%%%%%%%%%%%%%%%%%%%%%%%

\section{Conclusions and future work}\label{sec:conc}

We have provided necessary and sufficient conditions for the nonsingularity of $r$ coupled generalized Sylvester and $\star$-Sylvester equations \eqref{eq:gensystem}, with square coefficients of the same size $n\times n$. We have shown that, in the nonsingular case, the problem can be reduced to periodic systems having at most one generalized $\star$-Sylvester equation. A characterization for the nonsingularity of periodic systems of just generalized Sylvester equations was obtained in an unpublished work by Byers and Rhee \cite{byers-rhee}. That characterization was given in terms of spectral properties of matrix pencils constructed from the coefficients of the system. We have provided an analogous characterization for the nonsingularity of periodic systems with exactly one generalized $\star$-Sylvester equation. We have also provided a characterization for both types of periodic systems (namely, the one with exactly one 
generalized $\star$-Sylvester equation and the one with only generalized Sylvester equations) in terms of spectral properties of formal products constructed from the coefficients of the system.  
Finally, we have presented an $O(n^3r)$ algorithm for computing the unique solution of a nonsingular system, which has been shown to be backward stable.

A future research line that naturally arises from this
work is to get a characterization of nonsingularity in the more
general setting of rectangular coefficients. Other possible generalizations, pointed out by the referees, include systems involving complex conjugation of the unknowns, like those considered in \cite{ref1a,ref1b}, or systems of periodic type \cite{ref3a,ref3b}.

\section*{Acknowledgments}
We wish to thank the anonymous referees for their comments that helped us to improve the presentation.

\bibliographystyle{abbrv}
\bibliography{biblio}
	
\end{document}